\documentclass{amsart}

\usepackage{amsmath}
\usepackage{amssymb}
\usepackage{amsthm}
\usepackage[all]{xypic}
\usepackage{fullpage}
\usepackage{verbatim}
\usepackage[normalem]{ulem}
\usepackage[usenames]{color}
\usepackage{url}
\usepackage[backref=page]{hyperref}

\numberwithin{equation}{section}
\theoremstyle{plain}
\newtheorem{theorem}[equation]{Theorem}
\newtheorem{proposition}[equation]{Proposition}
\newtheorem{corollary}[equation]{Corollary}
\newtheorem{lemma}[equation]{Lemma}

\theoremstyle{definition}
\newtheorem{defn}[equation]{Definition}

\newtheorem{claim}[equation]{Claim}
\newtheorem{convention}[equation]{Convention}

\theoremstyle{remark}

\newcommand{\beq}{\begin{equation}}
\newcommand{\eeq}{\end{equation}}

\newcommand{\blank}{\mbox{$\underline{\makebox[10pt]{}}$}}

\newcommand{\st}{\left\vert\right.}

\newcommand{\bbar}[1]{\overline{#1}}

\DeclareMathOperator{\Hom}{{Hom}}

\DeclareMathOperator{\Tor}{Tor}
\DeclareMathOperator{\Aut}{{Aut}}
\DeclareMathOperator{\Proj}{Proj}
\DeclareMathOperator{\Spec}{Spec}

\DeclareMathOperator{\reg}{reg}

\DeclareMathOperator{\im}{Im}

\DeclareMathOperator{\Skl}{Skl}

\DeclareMathOperator{\gr}{gr}

\DeclareMathOperator{\op}{op}
\DeclareMathOperator{\Ann}{Ann}
\DeclareMathOperator{\Sym}{Sym}
\DeclareMathOperator{\alg}{alg}

\newcommand{\sh}{\mathcal}
\newcommand{\mc}{\mathcal}
\newcommand{\mf}{\mathfrak}

\DeclareMathOperator{\shSym}{\mathcal{S}\!\mathit{ym}}
\renewcommand{\Lsh}{\mathcal{L}}

\newcommand{\struct}{\mathcal{O}}
\newcommand{\kk}{{\Bbbk}}

\newcommand{\ZZ}{{\mathbb Z}}

\newcommand{\PP}{{\mathbb P}}
\newcommand{\NN}{{\mathbb N}}

\newcommand{\wt}{\widetilde}

\newcommand{\sF}{\mc{F}}
\newcommand{\sG}{\mc{G}}
\newcommand{\sH}{\mc{H}}
\newcommand{\sI}{\mc{I}}
\newcommand{\sJ}{\mc{J}}
\newcommand{\sK}{\mc{K}}
\newcommand{\sL}{\mc{L}}
\newcommand{\sM}{\mc{M}}
\newcommand{\sN}{\mc{N}}
\newcommand{\sO}{\mc{O}}

\newcommand{\sR}{\mc{R}}
\newcommand{\sS}{\mc{S}}

\newcommand{\pp}{positivity parameter}

\DeclareMathOperator{\rgr}{gr-\!}

\DeclareMathOperator{\rGr}{Gr-\!}

\DeclareMathOperator{\lmod}{\!-mod}
\DeclareMathOperator{\lMod}{\!-Mod}

\DeclareMathOperator{\lqgr}{\!-qgr}

\DeclareMathOperator{\rQgr}{Qgr-\!}
\DeclareMathOperator{\rqgr}{qgr-\!}

\DeclareMathOperator{\rTors}{Tors-\!}

\newcommand{\Xt}{X_{\infty}}
\newcommand{\Zt}{Z_{\infty}}
\newcommand{\Yt}{Y_{\infty}}

\newcommand{\Vt}{V_{\infty}}
\newcommand{\on}{\operatorname}

\DeclareMathOperator{\Cosupp}{Cosupp}
\newcommand{\dra}{\dashrightarrow}
\newcommand{\hra}{\hookrightarrow}
\newcommand{\ssm}{\smallsetminus}
\DeclareMathOperator{\Supp}{Supp}

\title{Moduli spaces for point modules on na\"ive blowups}
\author{Thomas A. Nevins}
\address{Department of Mathematics\\University of Illinois at Urbana-Champaign\\Urbana, IL 61801 USA}
\email{nevins@illinois.edu}
\author{Susan J. Sierra}
\address{School of Mathematics, University of Edinburgh, Edinburgh EH9 3JZ, United Kingdom.}
\email{s.sierra@ed.ac.uk}
\begin{document}

\begin{abstract}
The {\em na\"ive blow-up algebras} developed by Keeler-Rogalski-Stafford \cite{KRS}, after examples of Rogalski \cite{R-generic}, are the first known class of connected graded
algebras that are noetherian but not strongly noetherian.  This failure of the strong noetherian property is intimately related to the failure of the point modules over
such algebras to behave well in families: puzzlingly, there is no fine moduli scheme for such modules, although  point modules correspond bijectively with the points of a projective variety $X$.  We give a geometric structure to this bijection and prove that the variety $X$ is a coarse moduli space for point modules.
We also describe the natural moduli stack $\Xt$ for {\em embedded point modules}---an analog of a ``Hilbert scheme of
one point''---as an infinite blow-up of $X$ and establish good properties of $\Xt$.   The natural map $\Xt\rightarrow X$ is thus a kind of ``Hilbert-Chow morphism of one point"
for the na\"ive blow-up algebra.
\end{abstract}

\maketitle


\section{Introduction}\label{INTRO}
 One of the important achievements of noncommutative projective geometry is the classification of noncommutative projective planes,  such as the 3-dimensional Sklyanin algebra $\Skl_3$, by Artin-Tate-Van den Bergh \cite{ATV1990}.  More formally, these are {\em Artin-Schelter regular algebras} of dimension 3:  noncommutative graded rings that are close analogs of a commutative polynomial ring in 3 variables; see \cite{SV} for a discussion.  
The key method of  Artin-Tate-Van den Bergh is to study {\em point modules}; that is, cyclic graded modules with the Hilbert series of a point in projective space.  Given a noncommutative projective plane $R$,
Artin-Tate-Van den Bergh describe a moduli scheme for its point modules.  This allows them to construct 
a homomorphism from $R$ to a well-understood ring,  providing  a first step in describing the structure of the noncommutative plane itself.

The techniques described above  work in a more general context. 
Let $\kk$ be an algebraically closed field; we assume $\kk$ is uncountable, although for some of the results quoted this hypothesis is unnecessary. 
 A $\kk$-algebra $R$ is said to be {\em strongly noetherian} if for any commutative noetherian $\kk$-algebra $C$, the tensor product $R\otimes_\kk C$ is
again noetherian.
By a general result of Artin-Zhang \cite[Theorem~E4.3]{AZ2001}, if $R$ is a strongly noetherian $\NN$-graded $\kk$-algebra, then
its point modules are parameterized by a projective scheme.   Rogalski-Zhang \cite{RZ2008}  used this result to extend the method of \cite{ATV1990} to 
  strongly noetherian connected graded $\kk$-algebras  that are generated in degree 1.  (An $\NN$-graded $\kk$-algebra $R$ is {\em connected graded} if $R_0 = \kk$.)  Their method constructs a map from the algebra to a twisted homogeneous coordinate ring (see Section~\ref{BACKGROUND} for definitions) on the scheme $X$ parameterizing point modules.  For example, Sklyanin algebras are strongly noetherian, and here $X$ is an elliptic curve.  The homomorphism here gives the well-known embedding of an elliptic curve in a noncommutative $\PP^2$.  

Although it was believed for a time that all connected graded noetherian algebras would be strongly noetherian, Rogalski \cite{R-generic} showed this was not the case.  His example was generalized in joint work with Keeler and Stafford \cite{KRS, RS-0} to give a geometric construction of a beautiful class of noncommutative graded algebras, known as {\em na\"ive blow-ups}, that are noetherian but not strongly noetherian.  Along the way, they  showed that point modules for na\"ive blow-ups---viewed as objects of noncommutative projective geometry, in a way we 
make precise below---cannot behave well in families: there is no fine moduli scheme of finite type for such modules.

In the present paper, we systematically develop the moduli theory of point modules for the na\"ive blow-ups $S$ of \cite{KRS, RS-0}.  Roughly speaking, we show that there
is an analog of a ``Hilbert scheme of one point on $\operatorname{Proj}(S)$'' that is an infinite blow-up of a projective variety.  This
infinite blow-up is quasi-compact and noetherian as an {\em fpqc-algebraic} stack (a notion we make precise in
Section \ref{STACKS}).  Furthermore, we show there is a {\em coarse} ``moduli space
for one point on $\operatorname{Proj}(S)$''---it is, in fact, the projective variety from which the na\"ive blowup was constructed.  These are the first descriptions in the literature of moduli structures for 
point modules on a na\"ive blow-up.

More precisely, let $X$ be a projective $\kk$-variety of dimension at
least $2$, let $\sigma$ be an automorphism of $X$, and let $\Lsh$ be a
$\sigma$-ample (see Section~\ref{BACKGROUND}) invertible sheaf on $X$.  We
follow the standard convention that
$\sL^{\sigma}:= \sigma^* \sL$.
 Let $P \in X$ (in the body of the paper we let $P$  be any 0-dimensional subscheme of $X$), and assume that the $\sigma$-orbit of  $P$ is {\em critically dense}: that is, it is infinite and every infinite subset is Zariski dense.  For $ n\geq 0$, let
\[ \sh{I}_n := \sh{I}_P \sh{I}_P^{\sigma} \cdots \sh{I}_P^{\sigma^{n-1}}
\;\;\; \text{and} \;\;\;
\sL_n := \sL \otimes \sL^{\sigma} \otimes \cdots \otimes \sL^{\sigma^{n-1}}.\]
Define 
$\sh{S}_n := \sh{I}_n \otimes \Lsh_n$
and let
\[S : = S(X, \sL, \sigma, P) := \bigoplus_{n \geq 0} H^0(X, \sh{S}_n).\]
The algebra $S$ is the na\"ive blow-up associated to the data $(X, \sL, \sigma, P)$.

 If  $\sL$ is sufficiently ample, then $S$ is generated in degree one; alternatively, a sufficiently large Veronese of $S$ is always generated in degree one.  
We will assume throughout that $S$ is generated in degree one.

A {\em point module} is a graded cyclic $S$-module $M$ with Hilbert series 
$1 + t + t^{2} + \cdots$.  We say $M$ is an {\em embedded point module} if we are given, in addition, a surjection $S \rightarrow M$
of graded modules.  Two embedded point modules $M$ and $M'$ 
are {\em isomorphic} if there
is an $S$-module isomorphism from $M$ to $M'$ that intertwines the maps from $S$.   

We begin by 
 constructing a moduli stack for embedded point modules.
 Recall that $\Xt$ is a {\em fine moduli space} (or stack) for embedded point modules if there is an $S$-module quotient $S\otimes_\kk {\mathcal O}_{\Xt} \rightarrow M$ that is a universal family for point modules: that is,  $M$ is an $\Xt$-flat family of embedded $S$-point modules, with the property
that, if $S\otimes_\kk C \rightarrow M'$ is any $C$-flat family of embedded point modules for a commutative $\kk$-algebra $C$, then there is a morphism $\operatorname{Spec}(C)\xrightarrow{f} \Xt$ and an isomorphism $f^*M \cong M'$ of families of 
embedded $S$-point modules.  
Let
$X_n $ be the blowup of $X$ at $\sI_n$; there is an inverse system $\dots \rightarrow X_n\rightarrow X_{n-1}\rightarrow\dots \rightarrow X$
of schemes.  Let $\Xt := \underset{\longleftarrow}{\lim}\, X_n$.
  This inverse limit exists as a stack.  More precisely,
in Definition \ref{fpqc alg stack}, we introduce the notion of an {\em fpqc-algebraic} stack.  We then have:

\begin{theorem}\label{ithm-fine}
The inverse limit $\Xt$ is a noetherian fpqc-algebraic stack.  The morphism $\Xt\rightarrow X$ is quasicompact.  
Moreover, $\Xt$ is a fine moduli space for embedded $S$-point modules.
\end{theorem}
We have been told that similar results were known long ago to M. Artin; however, they seem not to have been very
widely known even among experts, nor do they seem to have appeared in the literature.

Note that the stack $\Xt$ is discrete: its points have no stabilizers.  Thus, 
$\Xt$ is actually a {\em $\kk$-space} in the terminology of \cite{LMB}; in particular, this justifies our use of the phrase ``fine moduli space'' in the statement of the theorem.  However, $\Xt$ does not seem to have an \'etale cover by 
a scheme, and hence does not have the right to be called an algebraic space.

We recall that, by definition, the noncommutative projective scheme associated to $S$  is the quotient category
$\operatorname{Qgr-}S = \operatorname{Gr-}S/\operatorname{Tors-}S$ of graded right $S$-modules by the full subcategory of locally
bounded modules.  A {\em point object} in $\operatorname{Qgr-}S$ is the image of (a shift of) a point module.  If $S$ is a commutative graded algebra generated in degree 1, $\operatorname{Qgr-}S$ is equivalent to the category of quasicoherent sheaves on $\operatorname{Proj}(S)$; this justifies thinking of $\operatorname{Qgr-}S$ as the   noncommutative analog of a
projective scheme.

If $R$ is strongly noetherian and generated in degree 1, then a result of Artin-Stafford \cite[Theorem~10.2]{KRS} shows that point objects of $\rQgr R$ are parameterized by the same projective scheme $X$ that parameterizes embedded point modules.  On the other hand, for na\"ive blowups $S= S(X, \sL, \sigma, P)$ as above, 
Keeler-Rogalski-Stafford show:
\begin{theorem}[\cite{KRS}, Theorem~1.1] 
The algebra $S$ is noetherian but not strongly noetherian.  
Moreover, there is no fine moduli scheme of finite type over $\kk$ parameterizing point objects of $\rQgr S$.
\end{theorem}

By contrast, \cite{KRS} gives a simple classification (that fails in families), namely that point objects are in bijective correspondence with points of $X$:  to a point $x\in X$ we associate the $S$-module $\bigoplus H^0(X, \kk_x \otimes \sL_n)$.  
In the present paper,
we explain how these two facts about point objects of $\rQgr S$ naturally fit together.

Assume that $\sL$ is sufficiently ample (in the body of the paper we work with any $\sigma$-ample $\sL$ by considering shifts of point modules).  Let $F$ be the moduli functor of embedded  point modules over $S$.   
 Define an equivalence relation $\sim$ on $F(C)$ by saying that $M \sim N$ if (their images) are isomorphic in $\operatorname{Qgr-}S\otimes_\kk C$.  We obtain  a functor $G: \text{Affine schemes} \longrightarrow \text{Sets}$
by sheafifying (in the fpqc topology) the presheaf $G^{\operatorname{pre}}$ of sets defined by
$\operatorname{Spec} C \mapsto  F(C)/\sim$.

A scheme $Y$ is a {\em coarse moduli scheme} for point objects if it corepresents the functor $G$: that is, there is 
a natural transformation $G\rightarrow \operatorname{Hom}_{\kk}(-, Y)$ that is universal for natural transformations from $G$ to 
schemes.  

Our main result is:

\begin{theorem}\label{ithm-coarse}
The variety $X$ is a coarse moduli scheme for point objects in $\operatorname{Qgr-}S$.  
\end{theorem}
This gives a geometric structure to the bijection discovered by Keeler-Rogalski-Stafford.  

\begin{corollary}\label{intro cor}
There is a fine moduli space $\Xt$ for  embedded $S$-point modules but only a coarse moduli scheme $X$ for point objects of $\operatorname{Qgr-}S$.
\end{corollary}

It may be helpful to compare the phenomenon described by Corollary \ref{intro cor} to a related, though quite 
different, commutative phenomenon.  Namely, let $Y$ be a smooth projective (commutative) surface.  Fix $n\geq 1$.  Let
$R= {\mathbb C}[Y]$ denote a homogeneous coordinate ring of $Y$ (associated to a sufficiently ample invertible sheaf on $Y$), and consider graded quotient modules $R \rightarrow M$ such that $\on{dim}\, M_\ell  = n$ for $\ell\gg 0$.  
 By a general theorem of Serre, the moduli space for such quotients is the 
{\em Hilbert scheme of $n$ points on $Y$}, denoted $\operatorname{Hilb}^n(Y)$.  This is a smooth projective variety of dimension
$2n$.  Alternatively, remembering only the corresponding objects $[M]$ of $\operatorname{Qgr-}R\simeq \operatorname{Qcoh}(Y)$,
and imposing the further $S$-equivalence relation
(see Example 4.3.6 of \cite{HL}), we get the moduli space 
$\operatorname{Sym}^n(Y)$ for semistable length $n$ sheaves on $Y$, which equals the $n$th symmetric product of $Y$.  The latter
moduli space is only a coarse moduli space for semistable sheaves.  One has the Hilbert-Chow morphism 
$\operatorname{Hilb}^n(Y)\rightarrow \operatorname{Sym}^n(Y)$ which is defined by taking a quotient
$R \rightarrow M$ to the equivalence class of $M$.  It is perhaps helpful to view the moduli spaces and map $\Xt\rightarrow X$ associated to the algebra $S$ in light of
the theorems stated above: that is, as a kind of ``noncommutative Hilbert-Chow morphism of one point'' for a na\"ive blow-up algebra 
$S(X, \sL, \sigma, P)$.  

In work in preparation, we  generalize the results in \cite{RZ2008} by proving a converse, of sorts, to Theorem~\ref{ithm-coarse}.  
Namely, suppose $R$ is a 
connected graded noetherian algebra generated in degree one;  that $R$ has a fine moduli space $\Xt$ for embedded point modules; that $R$ has a projective coarse moduli scheme $X$ for point objects of $\operatorname{Qgr-}R$; and that the spaces $\Xt$ and $X$ and the
morphism 
$\Xt\rightarrow X$ between them have geometric properties similar to those of the spaces we encounter in the theorems above.  
Then, we show, there exist an automorphism $\sigma$ of $X$, a zero-dimensional subscheme $P\subset X$ supported on points with critically dense orbits, 
an ample and $\sigma$-ample invertible sheaf ${\mathcal L}$ on $X$, and a homomorphism $\phi: R\rightarrow S(X,{\mathcal L}, \sigma, P)$ from $R$ to
the na\"ive blow-up associated to this data; furthermore, $\phi$ is surjective in large degree.  This construction gives a new tool for
analyzing the structure of rings that are noetherian but not strongly noetherian.  Details will appear in \cite{NS2}.  


{\bf Acknowledgements.}  
The authors are grateful to B. Conrad, J.T. Stafford, and M. Van den Bergh for helpful conversations;  to S. Kleiman for help with references; and to the referee for several helpful questions and comments.  

The first author was supported by NSF grant DMS-0757987.  The second author was supported by an NSF Postdoctoral Research
Fellowship, grant DMS-0802935.

\section{Background}\label{BACKGROUND}

 In this section, we give needed definitions and background.  We begin by discussing {\em   bimodule algebras}:  this is the correct way to think of the sheaves $\sS_n$ defined above.  Most of the material in this section was developed in \cite{VdB1996} and  \cite{AV}, and we refer the reader there for references.  Our presentation  follows that in \cite{KRS} and \cite{S-idealizer}.
 
\begin{convention}
Throughout the paper, by {\em variety} (over $\kk$) we mean an integral separated scheme of finite type over $\kk$.
\end{convention}
   
Throughout this section, let $\kk$ be an algebraically closed field and let $A$ denote an affine noetherian $\kk$-scheme, which we think of as a base scheme.

\begin{defn}\label{def-bimod}
Let $X$ be a  scheme of finite type over  $A$.  An {\em $\struct_X$-bimodule}  is  a quasicoherent $\struct_{X \times X}$-module $\sh{F}$, such that for every coherent submodule $\sh{F}' \subseteq \sh{F}$,  the projection maps $p_1, p_2: \Supp \sh{F}' \to X$ are both finite morphisms.    The left and right $\struct_X$-module structures associated to an $\struct_X$-bimodule $\sh{F}$ are defined respectively as $(p_1)_* \sh{F}$ and $ (p_2)_* \sh{F}$.  
We make the notational convention that when we refer to an $\struct_X$-bimodule simply as an $\struct_X$-module, we are using the left-handed structure (for example, when we refer to the global sections or higher cohomology
 of an $\struct_X$-bimodule).   All $\sO_X$-bimodules are assumed to be $\sO_A$-symmetric.

There is a tensor product operation on the category of bimodules that has the expected properties; see \cite[Section~2]{VdB1996}.
\end{defn}

All the bimodules that we consider will be constructed from bimodules of the following form:
\begin{defn}\label{def-LR-structure}
Let $X$ be a projective scheme over $A$ and let $\sigma, \tau \in \Aut_A(X)$. Let $(\sigma, \tau)$ denote the map
\begin{displaymath}
X  \to X \times_A X \;\;\; \text{defined by} \;\;\;
x  \mapsto (\sigma(x), \tau(x)).
\end{displaymath}
If  $\sh{F}$ is a quasicoherent sheaf on $X$, we define the $\struct_X$-bimodule ${}_{\sigma} \sh{F}_{\tau}$ to be 
${}_{\sigma} \sh{F}_{\tau} = (\sigma, \tau)_* \sh{F}.$
If $\sigma = 1$ is the identity, we will often omit it; thus we write $\sh{F}_{\tau}$ for ${}_1 \sh{F}_{\tau}$ and $\sh{F}$ for  the  $\struct_X$-bimodule ${}_1 \sh{F}_1 = \Delta_* \sh{F}$, where $\Delta: X \to X\times_A X$ is the diagonal.
\end{defn}

\begin{defn}\label{def-BMA}
Let $X$ be a projective scheme over $A$.  
An {\em $\struct_X$-bimodule algebra}, or simply a {\em bimodule algebra}, $\sh{B}$ is an algebra object in the category of bimodules.  That is, there are a unit map $1: \struct_X \to \sh{B}$ and a product map $\mu: \sh{B} \otimes \sh{B} \to \sh{B}$ that have the usual properties.
\end{defn}

We follow \cite{KRS} and define
\begin{defn}\label{def-gradedBMA}
Let $X$ be a projective scheme over $A$ and let $\sigma \in \Aut_A(X)$.  
A bimodule algebra $\sh{B}$ is a {\em graded $(\struct_X, \sigma)$-bimodule algebra} if:

(1) There are coherent sheaves $\sh{B}_n$ on $X$ such that
$\sh{B} = \bigoplus_{n \in \ZZ} {}_1(\sh{B}_n)_{\sigma^n};$

(2)  $\sh{B}_0 = \struct_X $;

(3) the multiplication map $\mu$ is given by $\struct_X$-module maps
$\sh{B}_n \otimes \sh{B}_m^{\sigma^n} \to \sh{B}_{n+m}$, satisfying the obvious associativity conditions.  
\end{defn}

\begin{defn}\label{def-module}
Let $X$ be a projective  scheme over $A$ and let $\sigma \in \Aut_A(X)$.  
Let $\sh{R} = \bigoplus_{n\in\ZZ} (\sh{R}_n)_{\sigma^n}$ be a graded $(\struct_X, \sigma)$-bimodule algebra.  A {\em right $\sh{R}$-module} $\sh{M}$ is a quasicoherent $\struct_X$-module $\sh{M}$ together with a right $\struct_X$-module map $\mu: \sh{M} \otimes \sh{R} \to \sh{M}$ satisfying the usual axioms.  We say that $\sh{M}$ is {\em graded} if there is a direct sum decomposition 
$\displaystyle \sh{M} = \bigoplus_{n \in \ZZ}  (\sh{M}_n)_{\sigma^n}$
 with multiplication giving a family of $\struct_X$-module maps $\sh{M}_n \otimes \sh{R}_m^{\sigma^n} \to \sh{M}_{n+m}$, obeying the appropriate axioms.
 
  We say that $\sh{M}$ is {\em coherent} if there are a coherent $\struct_X$-module $\sh{M}'$ and a surjective map $\sh{M}' \otimes \sh{R} \to \sh{M}$ of ungraded $\sR$-modules.  
We make similar definitions for left $\sh{R}$-modules.  The bimodule algebra $\sh{R}$ is {\em right (left) noetherian}
if every right (left) ideal of $\sh{R}$ is coherent.  A graded $(\struct_X, \sigma)$-bimodule algebra is right (left) noetherian if and only if every graded right (left) ideal is coherent.
\end{defn}

We recall here some standard notation for module categories over rings and bimodule algebras.  Let $C$ be a commutative ring and let $R$ be an $\NN$-graded $C$-algebra.  We define $\rGr R$ to be the category of $\ZZ$-graded right $R$-modules; morphisms in $\rGr R$ preserve degree.   Let $\rTors R$ be the full subcategory of modules that are direct limits of right bounded modules.  This is a Serre subcategory of $\rGr R$, so we may form the {\em quotient category}
\[ \rQgr R := \rGr R / \rTors R.\]
(We refer the reader to \cite{Gabriel1962} as a reference for the category theory used here.)  There is a canonical  quotient functor from $\rGr R$  to $ \rQgr R$. 

We make similar definitions on the left.   Further, throughout this paper, we adopt the convention that if Xyz is a category, then xyz is the full subcategory of noetherian objects.  Thus we have $\rgr R$ and $\rqgr R$, $R \lqgr$, etc.  If $X$ is  a scheme, we will denote the category of quasicoherent (respectively coherent) sheaves on $X$ by $\struct_X \lMod$ (respectively $\struct_X \lmod$).  

Given a module $M \in \rgr R$, we define $M[n] := \bigoplus_{i \in \ZZ} M[n]_i$, where
$M[n]_i = M_{n+i}.$

For a graded $(\struct_X, \sigma)$-bimodule algebra $\sh{R}$, we likewise define $\rGr \sh{R}$ and $\rgr \sh{R}$.  The full subcategory  $\rTors \sh{R}$ of $\rGr \sh{R}$ consists of direct limits of modules that are  coherent as $\struct_X$-modules, and we similarly define
$\rQgr \sh{R} := \rGr \sh{R}/\rTors \sh{R}.$
We define $\rqgr \sh{R}$ in the obvious way.

If $\sh{R}$ is an $\struct_X$-bimodule algebra, its global sections $H^0(X, \sh{R})$ inherit an $\sO_A$-algebra structure.  We call $H^0(X, \sR)$ the {\em section algebra} of $\sR$.  If $\sR = \bigoplus (\sR_n)_{\sigma^n}$ is a graded $(\sO_X, \sigma)$-bimodule algebra, then multiplication on $H^0 (X, \sh{R})$ is induced from  the maps
\[ \xymatrix{
H^0(X, \sh{R}_n) \otimes_A H^0(X, \sh{R}_m) \ar[r]^{1\otimes \sigma^n} 
&  H^0(X, \sh{R}_n) \otimes_A H^0(X, \sh{R}_m^{\sigma^n}) \ar[r]^<<<<<{\mu} 
& H^0(X, \sh{R}_{n+m}).
}\]

  If $\sh{M}$ is a graded right $\sh{R}$-module, then 
$H^0(X, \sh{M}) = \bigoplus_{n \in \ZZ} H^0(X, \sh{M}_n)$
 is a right $H^0(X, \sh{R})$-module in the obvious way; thus $H^0(X, \blank)$ is a functor from $\rGr \sh{R}$ to $\rGr H^0(X, \sh{R})$.  

 If  $R = H^0(X, \sh{R})$ and $M$ is  a graded right $R$-module, define 
$M \otimes_R \sh{R}$ to be the sheaf associated to the presheaf $V \mapsto M \otimes_R \sh{R}(V)$.   This is  a graded right $\sh{R}$-module, and the functor $\blank \otimes_R \sh{R}: \rGr R \to \rGr \sh{R}$ is a right adjoint to $H^0(X, \blank)$.

The following is a relative version of a standard definition.

\begin{defn}\label{def-ample}
Let $A$ be an affine $\kk$-scheme and let $q:  X \to A$ be a projective morphism.  Let $\sigma \in \Aut_A(X)$, and let  $\{\sh{R}_n\}_{n \in \NN}$ be a sequence of coherent sheaves  on $X$.  The sequence of bimodules $\{(\sh{R}_n)_{\sigma^n}\}_{n \in \NN}$ is {\em right ample} if for any coherent $\struct_X$-module $\sh{F}$, the following properties hold:
\begin{enumerate} 
\item  $\sh{F} \otimes \sh{R}_n$ is globally generated for $n \gg 0$ --- that is, the natural map $q^* q_*( \sF \otimes \sR_n )\to \sF \otimes \sR_n$ is surjective for $n \gg 0$;
\item $R^iq_*( \sh{F} \otimes \sh{R}_n )= 0$  for $n \gg 0$ and $i \geq 1$.  
\end{enumerate} 
\noindent The sequence $\{ (\sh{R}_n)_{\sigma^n}\}_{n \in \NN}$  is {\em left ample}  if for any coherent $\struct_X$-module $\sh{F}$, the following properties hold:
\begin{enumerate}
\item the natural map $q^* q_* ( \sR_n  \otimes \sF^{\sigma^n} )\to \sR_n  \otimes \sF^{\sigma^n}$ is surjective for $n \gg 0$; 
\item $R^i q_* (\sh{R}_n \otimes \sh{F}^{\sigma^n}) = 0$  for $n \gg 0$ and  $i \geq 1$.  
\end{enumerate}

If $A = \kk$, we say that an invertible  sheaf $\Lsh$ on $X$ is {\em $\sigma$-ample}
 if the $\struct_X$-bimodules 
\[\{(\Lsh_n)_{\sigma^n}\}_{n \in \NN} = \{ \Lsh_\sigma^{\otimes n}\}_{n \in \NN}\]
 form a right ample sequence.  By  \cite[Theorem~1.2]{Keeler2000}, this is true if and only if the $\struct_X$-bimodules $\{(\Lsh_n)_{\sigma^n}\}_{n \in \NN}$ form a left ample sequence.
\end{defn}

The following result is a special case of a result due to Van den Bergh \cite[Theorem~5.2]{VdB1996}, although we follow the presentation of  \cite[Theorem~2.12]{KRS}:

\begin{theorem}[Van den Bergh]\label{thm-VdBSerre}
Let $X$ be a projective $\kk$-scheme  and let $\sigma$ be an automorphism of $X$.  Let $\sh{R} = \bigoplus (\sh{R}_n)_{\sigma^n}$ be a right noetherian graded $(\struct_X, \sigma)$-bimodule algebra, such that the bimodules  $\{(\sh{R}_n)_{\sigma^n}\}$ form a right ample sequence.   Then $R = H^0(X, \sh{R})$ is also right noetherian, and the functors $H^0(X, \blank)$ and $\blank \otimes_R \sh{R}$ induce an equivalence of categories  $\rqgr \sh{R} \simeq \rqgr R.$ 
\end{theorem}

{\em Castelnuovo-Mumford regularity} is a useful tool for measuring ampleness and studying ample sequences.   
 We will need to use  relative  Castelnuovo-Mumford regularity; we review the relevant background here.
In the next three results, let $X$ be a projective $\kk$-scheme, and let $A$ be a noetherian $\kk$-scheme.  
Let $X_A:= X \times A$ and let $p: X_A \to X$ and $q:  X_A \to A$ be the projection maps.

  Fix a very ample invertible sheaf $\sO_X(1)$ on $X$.
Let $\sO_{X_A}(1):= p^* \sO_X(1)$; note $\sO_{X_A}(1)$ is relatively ample for $q:  X_A \to A$.  
If  $\sF$ is a coherent sheaf on $X_A$ and $n \in \ZZ$, let  $\sF(n) := \sF \otimes_{X_A} \sO_{X_A}(1)^{\otimes n}$.  
We say $\sF$ is  {\em $m$-regular with respect to $\sO_{X_A}(1)$}, or just {\em $m$-regular}, if 
$R^iq_*\sF(m-i) = 0$
for all $i >0$.
Since $\sO_{X_A}(1)$ is relatively ample, $\sF$ is $m$-regular for some $m$.  
The  {\em regularity} of $\sF$ is the minimal $m$ for which $\sF $ is $m$-regular; we write it $\reg(\sF)$.  

Castelnuovo-Mumford regularity is usually defined only for $\kk$-schemes, so we will spend a bit of space on the technicalities of working over a more general base.  
First note that we have:
\begin{lemma}\label{lem-CM1}
Let $\sF$ be a coherent sheaf on $X$.  
Then 
$ \reg(\sF) = \reg(p^* \sF).$ \qed
\end{lemma}

The fundamental result on Castelnuovo-Mumford regularity is due to  Mumford:
\begin{theorem}\label{thm-Mumford}
{\rm (\cite[Example~1.8.24]{Laz})}
 Let $\sF$ be an $m$-regular coherent sheaf on $X_A$. Then for every $n \geq 0$:
\begin{enumerate}
 \item $\sF$ is $(m+n)$-regular;
\item $\sF(m+n)$ is generated by its global sections: that is, the natural map 
$q^* q_* \sF(m+n)\to \sF(m+n)$
 is surjective; 
\item the natural map
$q_* \sF(m) \otimes_A q_* \sO_{X_A}(n) \to q_* \sF(m+n)$
is surjective.
\end{enumerate}
\end{theorem}

We also have
\begin{lemma}\label{SL1}
For any $0$-regular invertible sheaf $\sH$ on $X_A$ and any $A$-point $y$ of $X$, the natural map
\[ q_* \sH \stackrel{\alpha}{\to} q_* (\sO_y \otimes_{X_A} \sH)\]
is surjective.
\end{lemma}
\begin{proof}
 This is standard, but we check the details.  
Since cohomology commutes with flat base change, it suffices to consider the case that $A = \Spec C$ where $C$ is a local ring.
Then for any $n \in \ZZ$, we may consider $\sO_y \otimes_{X_A} \sH(n)$ as an invertible sheaf on $A$.  
Since $C$ is local, as a $C$-module this is isomorphic to $C$.

We thus have $q_* (\sO_y \otimes_{X_A} \sH) \cong C$.  
Let $I:= \im(\alpha)$; this is an ideal of $C$.

Let $n \geq 0$ and consider the natural maps
\beq\label{square}
\xymatrix{
 q_*  \sH \otimes_C q_* \sO_{X_A}(n) \ar[rr]^{\mu} \ar@{..>}[d] \ar[rd]_{\alpha \otimes 1} \ar[rrd]^{f} &&  q_*  \sH(n) \ar[d] \\
I \otimes_C q_* \sO_{X_A}(n) \ar[r]& q_*(\sO_y \otimes_{X_A} \sH) \otimes_C q_* \sO_{X_A}(n) \ar[r]& q_* (\sO_y \otimes_{X_A} \sH(n)) \cong C.
} \eeq
This diagram clearly commutes, and by construction  $\alpha \otimes 1$ factors through $I \otimes_C q_* \sO_{X_A}(n)$.  
Thus $\im f \subseteq I$ for all $n$.

On the other hand, by Theorem~\ref{thm-Mumford}(3), $\mu$ is surjective.   
As $\sO_{X_A}(1)$ is relatively ample, for $n \gg 0$ the right hand vertical map is surjective.
Thus $f$ is surjective for $n \gg 0$, and so $I = C$.
\end{proof}

Let $Z$ be a closed subscheme of $X_A$. 
 We say that $Z$ has {\em relative dimension} $\leq d$ 
if for all $x \in A$, the fiber $q^{-1}(x)$ has dimension $\leq d$ as a $\kk(x)$-scheme.  

The following is a relative version of a result of Keeler.
\begin{proposition}\label{prop-Dennis}
 {\rm (cf. \cite[Proposition~2.7]{Keeler2006})}
Let $X$ be a projective $\kk$-scheme.  There is a constant $D$, depending only on $X$ and on $\sO_X(1)$, so that the following holds:
for any  noetherian $\kk$-scheme $A$,
and for any coherent sheaves $\sF, \sG$ on $X_A$ so that the closed subscheme   of $X_A$ where $\sF$ and $\sG$ both fail to be locally free has relative dimension $\leq 2$, we have
\[ \reg(\sF\otimes_{X_A} \sG) \leq \reg(\sF) + \reg(\sG) + D.\]
\end{proposition}
\begin{proof}
The statement is local on the base, so we may assume without loss of generality that $A = \Spec C$ is affine.  
Since standard results such as Theorem~\ref{thm-Mumford} and Lemma~\ref{SL1} hold in this relative context, we may repeat the proof of \cite[Proposition~2.7]{Keeler2006}.  
The relative dimension assumption ensures the   vanishing of $Rq_*$ that is needed in the proof.
\end{proof}

To end the introduction, we define {\em na\"ive blowups}:  these are the algebras and bimodule algebras that we will work with throughout the paper.  Let  $X$ be a projective $\kk$-variety.  Let $\sigma\in \Aut_{\kk}(X)$ and let $\sL$ be a $\sigma$-ample invertible sheaf on $X$.  Let $P$ be a 0-dimensional subscheme of $X$.  We define ideal sheaves
\[\sI_n := \sI_P  \sI_P^{\sigma} \cdots \sI_P^{\sigma^{n-1}}\]
for $n \geq 0$.
Then we define a bimodule algebra
$\sS(X, \sL, \sigma, P) := \bigoplus_{n \geq 0} (\sS_n)_{\sigma^n},$
where
$\sS_n := \sI_n \sL_n.$
Define
$S(X, \sL, \sigma, P):= H^0(X, \sS(X, \sL, \sigma, P)).$

We recall the main results of \cite{RS-0}:
\begin{theorem}[\cite{RS-0}, Theorems 1.2 and 3.1]\label{thm-RS}
Let $X$ be a projective $\kk$-variety with $\dim X \geq 2$.  Let $\sigma \in \Aut_{\kk}(X)$ and let $\sL$ be a $\sigma$-ample invertible sheaf on $X$.  Let $P$ be a 0-dimensional subscheme of $X$.  Let $\sS:= \sS(X, \sL, \sigma, P)$ and let $S:= S(X, \sL, \sigma, P)$.  

If all points in $P$ have critically dense $\sigma$-orbits, then:
\begin{enumerate}
\item The sequence of bimodules $\{ (\sS_n)_{\sigma^n} \}$ is a left and right ample sequence.
\item $S$ and $\sS$ are left and right noetherian, and the categories $\rqgr S$ and $\rqgr \sS$ are equivalent via the global sections functor.  Likewise, $S \lqgr$ and $\sS \lqgr$ are equivalent.
\item The isomorphism classes of simple objects in $\rqgr S \simeq \rqgr \sS$ are in 1-1 correspondence with the closed points of $X$, where $x \in X$ corresponds to the $\sS$-module $\bigoplus \kk_x \otimes \sL_n$.  However, the simple objects in $\rqgr S$ are not parameterized by any scheme of finite type over $\kk$.
\end{enumerate}
\end{theorem}

For technical reasons, we will want to assume that our na\"ive blowup algebras $S$ is  generated in degree one.  
By \cite[Propositions~3.18, 3.19]{RS-0}, this will always be true if we either replace $S$ by a sufficiently large Veronese or replace $\sL$ by a sufficiently ample line bundle (for example, if $\sL$ is ample, by a sufficiently high tensor power of $\sL$).  
If $S$ is generated in degree one, then by \cite[Corollary~4.11]{RS-0} the simple objects in $\rqgr S$ are the images of shifts of point modules.

\section{Blowing up arbitrary 0-dimensional schemes}\label{GEOM-BLOWUPS}
For the rest of the paper, let $\kk$ be an uncountable algebraically closed field.  
Let $X$ be a projective variety over $\kk$, let $\sigma \in \Aut_{\kk}(X)$, and let $\sL$ be a $\sigma$-ample invertible sheaf on $X$.  Let $P$ be a 0-dimensional subscheme of $X$, supported at points with dense  (later, critically dense) orbits.  Let $\sS:=\sS(X, \sL, \sigma, P)$ and let $S:= S(X, \sL, \sigma, P)$.  
In this paper, we compare three objects:  the scheme parameterizing length $n$ truncated point modules over $S$, the scheme parameterizing length $n$ truncated point modules over $\sS$, and the blowup of $X$ at the ideal sheaf $\sI_n = \sI_P \cdots \sI_P^{\sigma^{n-1}}$.  In this section, we focus on the blowup of $X$.  
We first give some general lemmas on blowing up the defining ideals of 0-dimensional schemes.  These are elementary, but we give proofs for completeness.

Suppose that $X$ is a variety and that $f:  Y \to X$ is a surjective, projective morphism of schemes.  Let $\eta$ be the generic point of $X$.  We define
\[ Y^o := \bbar{f^{-1}(\eta)},\]
and refer to $Y^o$, by abuse of terminology, as the {\em relevant component} of $Y$.  In our situation, $f$ will always be generically one-to-one and  $Y^o$ will be irreducible, with $f|_{Y^o}$ birational onto its image.

\begin{lemma}\label{lem-ideals}
Let $A$ be a variety of dimension $\geq 2$.  
Let $\sh{I}$ be the ideal sheaf of a 0-dimensional subscheme of $A$, and let $\pi: X \to A$ be the blowup of $A$ at $\sh{I}$.  
Let  $W$ be the scheme parameterizing colength  1 ideals inside $\sh{I}$.  Let $\phi:  W \to A$ be the canonical morphism that sends an ideal $\sh{J}$ to the support of $\sh{I}/\sh{J}$.  Then there is a closed immersion $c:  X \to W$  that gives an  isomorphism between $X$ and $W^o$.  Further, the diagram 
\[ \xymatrix{
X \ar[rr]^{c} \ar[rd]_{\pi} && W \ar[ld]^{\phi} \\
& A
}
\]
commutes.   
\end{lemma}
\begin{proof}
Without loss of generality $A = \Spec C$ is affine; let $I := \sI(A)$.  We may identify  $W$ with $ \Proj \Sym_C(I)$ (see Proposition 2.2 of \cite{Kleiman90}); under this identification  $\phi:  W \to A$ is induced by the inclusion $C \hra \Sym_C(I)$.  There is a canonical surjective map of graded $C$-algebras $ \Sym_C(I) \to C \oplus \bigoplus_{n \geq 1} I^n$, which is the identity on $C$.  This induces a closed immersion $c:  X \to W$ with $\phi c = \pi$ as claimed.
Further, both $\pi: X\to A$ and $\phi:  W \to A $ are isomorphisms away from $\Cosupp \sI$.  Thus $c$ gives a birational closed immersion (and therefore an isomorphism) onto $W^o$.
\end{proof}

\begin{lemma}\label{lem-blowupmaps}
Let $A$ be a variety of dimension $\geq 2$, and let $\sh{I}$ and  $\sh{J}$ be ideal sheaves on $A$.  Let $\sh{K}:=\sh{IJ}$.  Define $i: X \to A$ to be the blowup of $A$ at $\sh{I}$, $j:  Y \to A$ to be the blowup of $A$ at $\sh{J}$, and $k:  Z \to A$ to be the blowup of $A$ at $\sh{K}$.  
\begin{enumerate}
\item[(a)]
 There are morphisms $\xi:  Z \to X$ and $\omega:  Z \to Y$ so that the diagram 
\[ 
\xymatrix{
Z \ar[r]^{\xi} \ar[dr]_{k} \ar[d]_{\omega} & X \ar[d]^{i} \\ 
Y \ar[r]_{j}	& A
} 
\]
commutes.

\item[(b)]
We have $Z \cong (X \times_A Y)^o$.

\item[(c)]
Let $W$ be the moduli scheme of subsheaves of $\sh{K}$ of colength 1, and let $V$ be the moduli scheme of subsheaves of $\sh{I}$ of colength 1.  Let $c:  Z \to W$ and $d:  X \to V$ be the maps from Lemma~\ref{lem-ideals}, and let $Z':= c(Z)$ and $X' := d(X)$.  Then the map 
$\xi':  Z' \to X'$
induced from $\xi$ sends $\sh{K}' \subset \sh{K}$ to $(\sh{K}':  \sh{J}) \cap \sh{I}$.
\end{enumerate}
\end{lemma}
\begin{proof}
(a).  Since $\xi^{-1}(\sh{K}) \struct_Z = \xi^{-1}(\sh{I}) \xi^{-1}(\sh{J}) \struct_Z$ is invertible, the inverse images of both $\sh{I}$ and $\sh{J}$ on $Z$ are invertible.  By 
 the universal property of blowing up (\cite[Proposition~7.14]{Hartshorne}), the morphisms $\xi: Z \to X$ and $\omega:  Z \to Y$ exist and commute as claimed.

For (b), (c) we may without loss of generality assume that $A = \Spec C$ is affine. 
 
 (b).  Let $U:= (X \times_A Y)^o$.  Let $A': =A \smallsetminus \Cosupp \sh{K}$.  Then 
 $U$ is the closure of $A'$ in $\PP^n_A \times_A \PP^m_A$ for appropriate $n, m$.

 Let $\phi:  \PP^n_A \times_A \PP^m_A \to \Sigma_{n,m} \subset \PP^{nm+m+n}_A$ be the Segre embedding.  Note that the canonical embeddings $Z \subseteq W \subseteq \PP^{nm+m+n}_A$ actually have $W \subseteq \Sigma_{n,m}$.   Since $\phi':= \phi|_U$ is the identity over $A'$ and $Z \subseteq \Sigma_{n,m}$ is the closure of $A'$ in $\PP^{nm+n+m}_A$, we have $\phi'(U) = Z$.

Let $p:  X \times_A Y \to X$ and $q:  X \times_A Y \to Y$ be the projection maps.  
From the commutative diagram in (a) we obtain a morphism $r:  Z \to X \times_A Y$ with $qr = \omega$ and $pr = \xi$.
Further, $r$ restricts to $(\phi')^{-1}$ over $A'$.  Thus $r (Z) = U$, and $\phi':  U \to Z$ is an isomorphism.

 (c).  A point $(x,y) \in \PP^n_A \times_A \PP^m_A$ corresponds to a pair of linear ideals $\mf{n} \subset C[x_0, \ldots, x_n]$ and $\mf{m} \subset C[y_0, \ldots, y_m]$.  Let $C[(x_i y_j)_{i,j}] \subset C[(x_i)_i][( y_j)_j]$ be the homogeneous coordinate ring of $\Sigma_{n,m}$.  It is clear that the ideal defining $\phi(x,y) = \{x\} \times \PP^m \cap \PP^n \times \{y\}$ in $C[x_i y_j]$ is generated by $\mf{n}_1 \cdot (y_0, \ldots, y_m) +  (x_0, \ldots, x_n) \cdot\mf{m}_1$.  

Let $(x,y) \in (X \times_A Y)^o$, where $x$ corresponds to the colength 1 ideal $\sI' \subseteq \sI$ and $y$ corresponds to $\sJ' \subseteq \sJ$.  
That $\sh{I}' \sh{J} + \sh{I} \sh{J}'$ gives the ideal $\sh{K}' \subset \sh{K}$ corresponding to $\phi(x,y)$ follows from the previous paragraph, together with the fact that the isomorphism $\phi'$ between $(X\times_A Y)^o$ and $Z$ is given by the Segre embedding.  

 Since $\phi'$ is an isomorphism, any ideal $\sh{K}'$ corresponding to a point $z \in Z$ may be written $\sh{K}' = \sh{I}' \sh{J} +  \sh{I} \sh{J}'$ for appropriate $\sh{I}'$, $\sh{J}'$.   We thus have
 $\sh{I}' \subseteq (\sh{K}' : \sh{J})\cap \sI \subsetneqq \sh{I}.$
 Since $\sh{I}'$ is colength 1, this implies that $\sh{I}' = (\sh{K}' : \sh{J})\cap \sI$, as claimed.  
 \end{proof}

\begin{corollary}\label{cor-X-maps}
Let $X$ be a projective variety of dimension $\geq 2$, let $\sigma \in \Aut_{\kk}(X)$, and let $\sI$ be an ideal sheaf on $X$.  Let $\sI_n:= \sI \sI^{\sigma} \cdots \sI^{\sigma^{n-1}}$.  For all $n\geq 0$, let 
$a_n: X_n \to X$ be the blowup of $X$ at $\sh{I}_n$.  Then there are birational morphisms
$\alpha_n:  X_n \to X_{n-1}$ (for $n \geq 1$) and $\beta_n:  X_n \to X_{n-1}$ (for $n \geq 2$) so that the diagrams
\[ \xymatrix{
X_n \ar[rr]^{\alpha_n} \ar[rd]_{a_n} && X_{n-1} \ar[ld]^{a_{n-1}} \\
& X & 
}
\;\;\;\; \text{and} \;\; \;\;
\xymatrix{
X_n \ar[r]^{\beta_n} \ar[d]_{a_n} & X_{n-1} \ar[d]^{a_{n-1}} \\
X \ar[r]_{\sigma} & X 
}\]
commute.
\end{corollary}
\begin{proof}
Let $\sh{K} := \sh{I}_n$.  Let $\zeta:  X_{n-1}'\to X$ be the blowup of $X$ at $\sh{I}_{n-1}^{\sigma}$.  Since $(\sh{I}_p)^{\sigma} \cong \sh{I}_{\sigma^{-1}(p)}$, there is an isomorphism
$\theta:  X_{n-1}' \to X_{n-1}$
so that the diagram
\[ \xymatrix{
X_{n-1}' \ar[r]^{\theta} \ar[d]_{\zeta} 	& X_{n-1} \ar[d]^{a_{n-1}} \\
X \ar[r]_{\sigma}  & X 
}
\]
commutes.   

Apply Lemma~\ref{lem-blowupmaps}(a) with $\sh{I} = \sh{I}_{n-1}^{\sigma}$ and $\sh{J} = \sh{I}_1$.  We obtain a morphism $\gamma:  X_n \to X_{n-1}'$ so that
\[ 
\xymatrix{
X_n \ar[r]^{\gamma} \ar[rd]_{a_n}	& X_{n-1}' \ar[d]^{\zeta} \ar[r]^{\theta} & X_{n-1} \ar[d]^{a_{n-1}} \\
& X \ar[r]_{\sigma}  & X 
}
\]
commutes.
Let $\beta_n:= \theta \gamma:  X_n \to X_{n-1}$.

Let $\alpha_n$ be the morphism $X_n \to X_{n-1}$ given by Lemma~\ref{lem-blowupmaps}(a) with $\sh{I} = \sh{I}_{n-1}$ and $\sh{J}= \sh{I}_1^{\sigma^{n-1}}$.  The diagram
\[ \xymatrix{
X_n \ar[r]^{\alpha_n} \ar[rd]_{a_n}	& X_{n-1} \ar[d]^{a_{n-1}} \\
& X 
}
\]
commutes, as required.
\end{proof}

We will frequently suppress the subscripts on the maps $\alpha_n, a_n$, etc., when the source and target are indicated. 
  Note that the equation
$a_n = \alpha_1 \circ \cdots \circ \alpha_n$
that follows from Corollary~\ref{cor-X-maps} may be written more compactly as
$ a = \alpha^n:  X_n \to X.$

\section{Infinite blow-ups}\label{STACKS}
In this section, we prove some general properties of infinite blow-ups that will be useful when we consider moduli spaces of 
embedded point modules.  Such infinite blow-ups can be handled in two ways: either as pro-objects in the category of schemes, or as stacks, 
via the (inverse) limits of such pro-objects in the category of spaces or of stacks.  We've chosen to treat infinite blow-ups as 
the limits rather than as pro-objects.  This is formally the correct choice, in the sense that the limit formally contains less information
than the pro-object.  We note that in our setting, we could also  work with the pro-object with no difficulties; however, we have found the language of stacks more natural.

We begin with some technical preliminaries on schemes and stacks.  
We will work with stacks in the fpqc (fid\`element plat et quasi-compact) topology; the fpqc topology of schemes is discussed in Section 2.3.2 of \cite{Vistoli}.  We are
interested in a class of stacks that are apparently not algebraic, but for which a certain amount of algebraic geometry is still possible.  
More precisely, recall that a stack ${\mathcal X}$ is called {\em algebraic} if 
the diagonal morphism of ${\mathcal X}$ is representable, separated and quasi-compact; and
 it has an fppf atlas $f: Z\rightarrow {\mathcal X}$ that is a scheme: that is, $f$ is representable, faithfully flat, and finitely presented.  By Artin's theorem \cite[Th\'eor\`eme~10.1]{LMB}
the second condition is equivalent to requiring the existence of a smooth, surjective and representable $f$.  

Our stacks
are very similar to algebraic stacks, but it seems not to be possible to find a finite-type $f$ for which $Z$ is a scheme.  
On the other 
hand, we can find $f$ for which $Z$ is a scheme and $f$ is fpqc---and even formally \'etale---so in some sense our stacks are the fpqc analogs
of algebraic stacks.  
\begin{defn}\label{fpqc alg stack}
We will refer to a stack ${\mathcal X}$ for which the diagonal 
$\Delta: {\mathcal X}\rightarrow {\mathcal X}\times_{\kk}{\mathcal X}$ is representable, separated, and quasi-compact, and which admits
a representable fpqc morphism $Z\rightarrow {\mathcal X}$ from a scheme $Z$, as {\em fpqc-algebraic}.  
\end{defn}
\noindent
Note that ``separated'' and ``quasi-compact'' make sense for fpqc stacks by \cite[IV.2~Proposition~2.7.1 and IV.2~Cor.~2.6.4]{EGA}.
Unfortunately, in this weaker setting, there are fewer notions of algebraic geometry that one can check fpqc-locally,
and hence fewer adjectives that one can sensibly apply to fpqc-algebraic stacks.  Still, one can make sense, for example, of 
representable morphisms being separated, quasi-separated, locally of finite type or of finite presentation, proper, closed
immersions, affine, etc. by \cite[IV.2~Proposition~2.7.1]{EGA}.

Recall \cite[IV.4~Def.~17.1.1]{EGA} that a morphism of schemes $f: X\rightarrow Y$ is {\em formally \'etale} if for every affine scheme $Y'$, 
closed subscheme $Y'_0\subset Y'$ defined by a nilpotent ideal and morphism $Y'\rightarrow Y$, the map
$\operatorname{Hom}_Y(Y', X)\rightarrow \operatorname{Hom}_Y(Y'_0, X)$ is bijective.    By faithfully flat descent (see 
\cite{Vistoli}), the definition extends immediately to 
stacks in the \'etale, fppf, and fpqc topologies of schemes.  

We will say that an fpqc-algebraic stack ${\mathcal X}$ is {\em noetherian} if it admits an fpqc atlas $Z\rightarrow {\mathcal X}$ by a
noetherian scheme $Z$.  Unfortunately, since fpqc morphisms need not be of finite type even locally, it does not seem to be possible
to check this property on an arbitrary atlas $Y\rightarrow {\mathcal X}$.

Suppose we have a sequence of schemes $\{ X_n \st n \in \NN\}$ and  projective morphisms $\pi_n:  X_n \to X_{n-1}$.  
We define the infinite blowup $\Xt$ to be the presheaf of sets $\Xt = \underset{\longleftarrow}{\lim}\, X_n$.  More precisely,  
we define the functor of points
$h_{\Xt}:  \text{Schemes}^{\mathrm{op}} \longrightarrow \text{Sets}.$
For each scheme  $A$, let 
\[ h_{\Xt}(A) =  \big\{(\zeta_n:  A \to X_n)_{n\in\NN} \;\big|\; \pi_n \zeta_n = \zeta_{n-1}  \big\} .\]
For each $n$, there is an induced map
$\pi:  \Xt \to X_n,$
where the target space $X_n$ is indicated explicitly.

\begin{proposition}\label{prop-algspace}
  Suppose that $X:= X_0$ is a variety of dimension $\geq 2$ and that there are  maps $\pi_n:  X_n \to X_{n-1}$ as above.
Then the stack $\Xt$ is a sheaf in the fpqc topology.  

Further, suppose that the maps $\pi_n$ satisfy the following conditions:
\begin{enumerate}
\item[(i)] For all $n$, $\pi^{-1}_n$ is defined at all but finitely many points of $X_{n-1}^o$.  That is, the set of exceptional points of $\pi^{-1}:  X \dra \Xt$ is countable; let $\{z_m\}_{m \in \NN}$ be an enumeration of this set.
 \item[(ii)] The set $\{z_m\}$ is critically dense.
\item[(iii)] For all $m$, there is some $n(m)$ so that, for $n \geq n(m)$, the map $\pi_n$ is a local isomorphism at all points in the preimage of $z_m$.   
\item[(iv)] For all $m \in \NN$, there is an ideal sheaf $\sJ_m$ on $X$, cosupported at $z_m$, so that  $X_{n(m)}$ is a closed subscheme of $\Proj \shSym_X \sJ_m$  above a neighborhood of $z_m$.  That is, $X_{n(m)} \to X$ factors as
\[ \xymatrix{ X_{n(m)} \ar[r]^(0.4){c_m} & \Proj \shSym_X \sJ_m \ar[r]^(0.7){p_m} & X}, \]
where $p_m: \Proj \shSym_X \sJ_m \to X$ is the natural map, and $c_m$ is a closed immersion over a neighborhood of $z_m$.  
\item[(v)] There is some $D \in \NN$
so that $\mathfrak{m}_{z_m}^D\mathcal{O}_{X,z_m}\subseteq \mathcal{J}_m\subseteq \mathcal{O}_{X,z_m}$
for every $m$.
\end{enumerate}
Then:  
\begin{enumerate}
\item The stack $\Xt$ is fpqc-algebraic:
\begin{enumerate}
\item[(1a)]  $\Xt$ has a representable, formally \'etale, fpqc cover by an affine scheme $U\rightarrow \Xt$.
\item[(1b)]  The diagonal morphism $\Delta: \Xt\rightarrow \Xt\times_{\kk}\Xt$ is representable, separated, and quasicompact.
\end{enumerate}
\item The morphism $\pi: \Xt\to X$ is quasicompact.
\item $\Xt$ is noetherian as an fpqc-algebraic stack.
\end{enumerate}
\end{proposition}
\begin{proof}
 Because any limit  of an inverse system of sheaves taken in the category of presheaves is already a
sheaf (cf. \cite[Exercise~II.1.12]{Hartshorne}), $\Xt$ is a sheaf in the fpqc topology.

Now assume that (i)--(v) hold.   
  For $n \in \NN$, let $W_n$ be the scheme-theoretic image of $c_n$.  Let
\[ X^\prime_n := W_0 \times_X W_1 \times_X \cdots \times_X W_{n-1}.\]
Let $\pi^{\prime}_n:  X^\prime_n \to X^\prime_{n-1}$ be projection on the first $(n-1)$ factors.   
   We  first show that we can assume without loss of generality that  $X_n = X_n^\prime$; that is, we claim that $\varprojlim_{\pi^\prime} X^\prime_n \cong \varprojlim_\pi X_n$.

Let $k \in \NN$.  Let $K(k) := \max\{k, n(0),\ldots, n(k-1)\}$. For each $0 \leq m \leq k-1$, there is a morphism 
\[ \xymatrix{ X_{K(k)} \ar[rr]^{\pi^{K(k)-n(m)}} && X_{n(m)} \ar[r]^{c_m} & W_m.}\]
Since these agree on the base, we obtain an induced
$ \phi_k:   X_{K(k)} \to X^\prime_k$.
The $\phi_k$ are clearly compatible with the inverse systems $\pi$ and $\pi^\prime$.  Taking the limit, we obtain
\[ \phi:  \varprojlim X_{K(k)} \to \varprojlim X^\prime_k.\]

Now let $F_k$ be the set of fundamental points of $X \dra X_k$.  Let
$N(k) := k + \max \{ m \st z_m \in F_k \}$.
We claim there is a morphism $\psi_k: X^\prime_{N(k)} \to X_k$.  There is certainly a rational map  defined over $X \ssm \{ F_k\}$, since there $X_k$ is locally isomorphic to $X$.  Let  $z_m \in F_k$ and let $n'(m):= \max\{k, n(m)\}$.  The rational map
\[ \xymatrix{ X^\prime_{N(k)} \ar[r] & W_m \ar@{-->}[r] & X_{n'(m)} \ar[r] & X_k}\]
is then defined over a neighborhood of $z_m$.  These maps clearly agree on overlaps, so we may glue to define $\psi_k$ as claimed. 
Let
\[ \psi:  \varprojlim X^\prime_{N(k)} \to \varprojlim X_k\]
be the limit of the $\psi_k$.  It is clear that $\psi = \phi^{-1}$; note that by construction both $N(k)$ and $K(k)$ go to infinity as $k \to \infty$.

Going forward, we replace $X_n$ by $X^\prime_n$.   Thus let $Y_n: = \Proj \shSym_X \sJ_n$, and assume that there are closed immersions $i_n:  X_n \to Y_0 \times_X \cdots \times_X Y_{n-1}$ so that the $\pi_n$ are given by restricting the projection maps.   

It suffices to prove the proposition in the case that $X = \Spec C$ is affine; note that we can choose an affine subset of $X$ that contains all $z_n$.  Let $J_n \subseteq C$ be the ideal cosupported at $z_n$ so that $(J_n)_{z_n} = \sh{J}_n$.  Let ${\mathfrak m}_p$ denote the maximal ideal of $C$ corresponding to $p$.  
\begin{claim}
There is an $N$ such that every ideal $J_m$ ($m\in \NN$) is generated by at most $N$ elements.
\end{claim} 
To prove the claim,  embed 
$X$ in an affine space, i.e. choose a closed immersion $X\subseteq \mathbb{A}^\ell$.  Then each point of $\mathbb{A}^\ell$, hence 
a fortiori each point of $X$, is cut out scheme-theoretically by $\ell$ elements of $C$, and the power of the maximal ideal $\mathfrak{m}_{z_m}^D$ appearing in hypothesis (v) of the proposition is generated by $N_0:=\binom{D+ \ell-1}{\ell-1}$ elements
of $C$.  Now $J_m$ contains $\mathfrak{m}_{z_m}^D$, and  
\begin{displaymath}
\on{dim}(J_m/\mathfrak{m}_{z_m}^D) \leq  \on{dim} C/\mathfrak{m}_{z_m}^D \leq \on{dim}\kk[u_1,\dots, u_\ell]/(u_1,\dots,u_\ell)^D =: N_1.
\end{displaymath}
  Thus $J_m$ is generated by at most $N:= N_0+N_1$ elements.

We continue with the proof of the proposition:

(1a)  To construct an affine scheme $U$ with a representable, formally \'etale morphism $U\rightarrow \Xt$ we proceed as follows.
 For each $n$, and for each $1 \leq i \leq N$, we choose hypersurfaces $D_{n,i} = V(d_{n,i}) \subset X$,  with the following properties:

\begin{enumerate}
\item[(A)] For all $n$,  the elements  $d_{n,1}, \ldots, d_{n,N}$ generate $J_n$.
\item[(B)] For all $n,i$, the hypersurface $D_{n,i}$ does not contain any  irreducible component $Z$ of a hypersurface $D_{m,j}$ with $m < n$ or $m =n$ and $j < i$.  
\item[(C)] For all $n$ and each $m\neq n$, $z_m\notin  D_{n,i}$ for any $i$.
\end{enumerate}
We can make such choices because $\kk$ is uncountable and $X$ is affine (note that in order to satisfy (B), the choice of each $D_{n,i}$ will depend on finitely many earlier choices).

For each $N$-tuple of positive integers $(n_1, \dots, n_N)$, let $Z_{(n_1,\dots, n_N)} := D_{n_1, 1}\cap \dots\cap D_{n_N, N}$.  Note that
$z_m\notin Z_{(n_1,\dots, n_N)}$ unless $(n_1,\dots, n_N) = (m, m, \dots, m)$ by property (C).  Note also that, since $Z_{(n_1,\dots, n_N)}$ is a union of intersections of pairwise distinct irreducible hypersurfaces, it has codimension at least $2$ in $X$.  

Now let $\displaystyle Z^{(1)} := \bigcup_{(n_1,\dots, n_N)} Z_{(n_1,\dots, n_N)}$.  This is a countable union of irreducible subsets of $X$ of codimension at least $2$.  We may choose one point lying on each component of $Z^{(1)} \smallsetminus \{z_m\}_{m \in \NN}$.  Now, for each $n$, choose a hypersurface $D_{n, N+1}$ such that $z_n\in D_{n,N+1}$; that the local ideal of $D_{n, N+1}$ at $z_n$ is contained in $\mathcal{J}_n$; and that 
$D_{n, N+1}$ avoids all the (countably many) chosen points of components of $Z^{(1)}$ and all $z_m, m\neq n$.  Then for each $n$,
$Z^{(1)}\cap D_{n,N+1}$ is a countable union of irreducible algebraic subsets of codimension at least $3$ (it is a union of proper intersections of $D_{n,N+1}$ with irreducible subsets of codimension at least $2$).  Let $\displaystyle Z^{(2)} := \bigcup_n (Z^{(1)}\cap D_{n, N+1})$.  

Repeating the previous construction with $Z^{(2)}$, we get hypersurfaces $D_{n, N+2}$ such that each $Z^{(2)}\cap D_{n, N+2}$ is a countable union of irreducible subsets of codimension at least $4$.  Iterating, we eventually define hypersurfaces 
$D_{n, N+i}$, $i=1, \dots, d$ with the following properties:
\begin{enumerate}
\item[(A$^\prime$)]  For all $m \in \NN$, there is a scheme-theoretic equality   $\operatorname{Spec}(C/J_m) = D_{m,1}\cap \dots \cap D_{m,N+d}$.
\item[(B$^\prime$)] For every sequence $(n_1, \dots, n_{N+d})$ of positive integers, we have a set-theoretic equality
\begin{displaymath}
D_{n_1,1}\cap \dots \cap D_{n_{N+d}, N+d} = \begin{cases} z_m & \text{if $(n_1,\dots, n_{N+d}) = (m, \dots, m)$.}\\
\emptyset & \text{otherwise.}\end{cases}
\end{displaymath}
\item[(C)] For all $n$ and each $m\neq n$, $z_m\notin  D_{n,i}$ for any $i$.
\end{enumerate}

For $0 \leq n \leq m-1$, we abusively let   $\wt{D}_{n,i} \subset X_m^o$ denote the proper transform of $D_{n,i}$.   By construction, $\wt{D}_{n,1} \cap \cdots \cap \wt{D}_{n, N+d} = \emptyset$.

For each $m\in \NN$, the  map $C^{N+d} \to J_m$, $e_i \mapsto d_{m,i}$ induces a closed immersion $Y_m \to \PP^{N+d-1}_X$.  Let $V_{m,i} \subseteq Y_m$ be the open affine given by $(e_i \neq 0)$.  Note that  $V_{m,i} \cap X_m^o = X^o_m \ssm \wt{D}_{m,i}$ (recall that here $X_m^o$ denotes the closure in $X_m$ of the preimage in $X_m$ of the generic point of $X$).   Let 
\[ U_{n,i} := X_n \cap (V_{0,i}\times_X V_{1,i} \times_X \cdots \times_X V_{n-1,i} ).\]
The $U_{n,i}$ are open and affine.   Since $D_{m,i} \not\ni z_n$ for $m \neq n$, the set $\bigcup_i U_{n,i}$ includes  all irreducible components of $X_n$ except possibly for $X_n^o$.  But
\[ X_n^o \ssm \bigcup_{i=1}^{N+d} U_{n,i} = \bigcap_i \bigcup_{m=0}^{n-1} \wt{D}_{m,i} = \bigcup_m \bigcap_i \wt{D}_{m,i} = \emptyset.\]
Thus the $U_{n,i}$ are an open affine cover of $X_n$.  

Since $\pi_n |_{U_{n,i}}$ is obtained by base extension from the affine morphism $V_{n-1,i} \to X$, it is affine, and $\pi_n(U_{n,i})\subseteq U_{n-1,i}$.  
  Writing $C_i  := \underset{\longrightarrow}{\lim}\, C_{m,i}$ and $U_i := \Spec C_i$, we get 
$U_i = \underset{\longleftarrow}{\lim}\, U_{m,i}$;
all the $U_i$ are affine schemes.
By construction we obtain induced maps $U_i \to \Xt$.   Let 
$\displaystyle U :=  \bigsqcup_i  U_i.$

\begin{claim}
 The induced morphism $U \to \Xt$ is representable and formally \'etale.  
 \end{claim}

Since each map $U_i\to\Xt$ is a
limit of formally \'etale morphisms, each is itself formally \'etale.  We must show that if $T$ is a scheme equipped with a morphism $T \to \Xt$, then $T \times_{\Xt} U \to T$ is a scheme over $T$.  Each morphism $T\times_{X_m} U_{m,i}\rightarrow 
T$ is an affine open immersion since the morphisms $U_{m,i}\rightarrow X_m$ are affine open immersions.  
Hence  the morphism 
$\underset{\longleftarrow}{\lim}\, (T\times_{X_m}U_{m,i})\rightarrow T$ is an inverse limit of schemes affine over $T$ and thus is itself a scheme affine
over $T$ (see \cite[IV, Proposition~8.2.3]{EGA}).  The claim now follows from:  
\begin{lemma}\label{equality of fiber products}
For any scheme $T$ equipped with a morphism
$T\rightarrow \Xt$, we have $T\times_{\Xt} U_i \cong \underset{\longleftarrow}{\lim}\, (T\times_{X_m}U_{m,i})$.  \qed
\end{lemma}

\begin{claim}\label{surjective}
The  map $U\rightarrow \Xt$ is surjective.
\end{claim}
Surjectivity for representable morphisms can be checked locally on the target by \cite[3.10]{LMB}.  Thus, we may change base along a map $T\rightarrow\Xt$ from a scheme $T$; and, taking a point that is the image of a map $\operatorname{Spec}(K)\rightarrow T$ where $K$ is a field containing $\kk$, it suffices to find  $\operatorname{Spec}(K) \rightarrow U$ making
\begin{equation}\label{lifting points}
\xymatrix{\operatorname{Spec}(K) \ar@{=}[d]\ar[r] &  U\ar[d]\\
\operatorname{Spec}(K) \ar[r] & \Xt}
\end{equation}
commute.

Thus, suppose we are given a map $\operatorname{Spec}(K)\rightarrow \Xt$; let $y_n$ denote its image (i.e. the image of the unique point of $\operatorname{Spec}(K)$) in $X_n$. 
Let $I_n$ denote the (finite) set of those $i$ so that $y_n \in U_{n,i}$.  Since the $U_{n,i}$ cover $X_n$, each $I_n$ is nonempty; further,  $I_n \subseteq I_{n-1}$.  The intersection $\bigcap_n I_n$ is thus nonempty and contains some $i_0$.
 The maps $\operatorname{Spec}(K) \rightarrow U_{m,i_0}$ for $m \gg 0$ define a map $\displaystyle f: \operatorname{Spec}(K) \rightarrow U_{i_0} = \lim_{\longleftarrow} U_{m,i_0} \subset U$, and thus defining the map in the top row of 
\eqref{lifting points} to be $f$ gives the desired commutative diagram.
This proves the claim.

Returning to the proof of Proposition \ref{prop-algspace}(1a), let $R := U \times_{\Xt} U$.  If we define
$R_{ij} := U_i \times_{\Xt} U_j,$
then we have 
$\displaystyle R = \bigsqcup_{i,j} R_{ij}.$
Note that $R_{ij}$ is a scheme affine over $U_i$ by the previous paragraph.  Since affine schemes are quasicompact, this proves that the morphism $U\to\Xt$ is quasicompact.  Furthermore, ${\mathcal O}(R_{ij})$ is a localization of $C_i$ (obtained by inverting the images of the elements $d_{n,j}$);
so $R_{ij}\to U_i$ is flat.  We have already proved that $U\to\Xt$ is surjective, so we conclude that $U\to\Xt$ is faithfully flat.  It follows that $U\to\Xt$ is 
fpqc using \cite[Proposition~2.33(iii)]{Vistoli}.
This completes the
proof of (1a).

(1b) The diagonal $\Delta: \Xt\rightarrow \Xt\times_{\kk} \Xt$ is the inverse limit of the diagonals 
$\Delta_n: X_n\rightarrow X_n\times_{\kk} X_n$.  Similarly to Lemma \ref{equality of fiber products}, if
$V\rightarrow \Xt\times_{\kk}\Xt$ is any morphism from a scheme $V$, we get 
$\Xt\times_{\Xt\times_{\kk}\Xt} V \cong  \underset{\longleftarrow}{\lim}\; X_n\times_{X_n\times_{\kk}X_n} V$.  Since each $X_n$
is separated over $\kk$, each morphism $X_n\times_{X_n\times_{\kk}X_n} V\rightarrow V$ is a closed immersion; hence
$\Xt\times_{\Xt\times_{\kk}\Xt} V\rightarrow V$ is a closed immersion.  This proves (1b).

(2) Again, we may assume that $X$ is affine.  Then, as above, we have an fpqc cover $U \xrightarrow{p} \Xt$ by an affine scheme $U$.  Since an 
affine scheme is quasicompact and a continuous image of a quasicompact space is quasicompact, $\Xt$ is quasicompact, as desired.

(3)  By our definition, it suffices to prove that $U$ is noetherian, or, equivalently, that each $C_i$ is a noetherian ring.    This follows as in \cite[Theorem~1.5]{ASZ1999}.  
We will need:
\begin{lemma}\label{lem-proj}
 Let $A$ be a commutative noetherian ring, and let $J$ be an ideal of $A$ with a resolution
\[  \xymatrix{ A^m \ar[r]^M & A^n \ar[r] & J \ar[r] & 0.}\]
(Here $M$  is an $n\times m$ matrix acting by left multiplication.)  Let $A':= A[t_1, \ldots, t_{n-1}]/{(t_1, \ldots, t_{n-1},1) M}$.  Let $P'$ be a prime of $A'$, and let $P:= P' \cap A$.  If $P$ and $J$ are comaximal, then $P A' = P'$.
\end{lemma}
Note that $A'$ is the coordinate ring of a  chart of $\Proj\ \Sym_A J$.  
\begin{proof}
We may localize at $P$, so without loss of generality $J = A$.  Then $A' \cong A[g^{-1}]$ for some $g \in A$.  The result follows. 
\end{proof}

We return to the proof of (3).  
  It suffices to show, by \cite[Exc. 2.22]{Eis}, that each prime of $C_i$ is finitely generated.  Let $\wt{P} \neq 0$ be a prime of $C_i$.  Let $P:= \wt{P}\cap C$, and let $P_n:= \wt{P} \cap C_{n,i}$.  By critical density, there is some $n \in \NN$ so that if $m \geq n$, then $J_m$ and $P$ are comaximal.  It follows from Lemma~\ref{lem-proj} that $P_m = P_n C_{m,i}$ for $m \geq n$.  So $\wt{P} = \bigcup P_m = P_n C_i$.  This is finitely generated because  $C_{n,i}$ is noetherian, so $P_n$ is finitely generated.   

Proposition~\ref{prop-algspace} is now proved.
\end{proof}

\begin{corollary}\label{cor-Xt}
Let $X$ be a projective variety, let $\sigma \in \Aut_{\kk}(X)$, and let $\sL$ be a $\sigma$-ample invertible sheaf on $X$.  Let $P$ be a 0-dimensional subscheme of $X$, all of whose points have critically dense $\sigma$-orbits.  Let  $\sI_n:= \sI_P \sI_P^{\sigma} \cdots \sI_P^{\sigma^{n-1}}$.  Let $a_n: X_n \to X$ be the blowup of $X$ at $\sI_n$, as in Corollary~\ref{cor-X-maps}.
Let $\alpha_n:  X_n \to X_{n-1}$ be given by Corollary~\ref{cor-X-maps}.  Then the limit
\[ \Xt := \varprojlim X_n\]
is a noetherian fpqc-algebraic stack.
\end{corollary}
\begin{proof}
This follows immediately from Proposition~\ref{prop-algspace}.
\end{proof}

\section{Moduli schemes for truncated point modules}\label{PTSCHEMES}
Let $X$ be a  projective variety, let $\sigma \in \Aut_{\kk}(X)$, and let $\sL$ be a $\sigma$-ample invertible sheaf on $X$.  Let $P$ be a 0-dimensional subscheme of $X$, all of whose points have critically dense $\sigma$-orbits.    We define
$\sS:= \sS(X, \sL, \sigma, P)$
and
$S:= S(X, \sL, \sigma, P),$
as in Section~\ref{BACKGROUND}.  
As usual, we assume that $S$ is generated in degree one.

In this section, we construct moduli schemes of truncated point modules over $\sh{S}$ and $S$.  
In the next section, we compare them.  
We begin by constructing moduli schemes for shifted point modules for an arbitrary connected graded noetherian algebra generated in degree 1, generalizing slightly results of  \cite{ATV1990} and \cite{RS}. 

Let $C$ be any commutative $\kk$-algebra.  
Recall that we use subscript notation to denote changing base.  Thus if $R$ is a $\kk$-algebra,  we write
$R_C:= R\otimes_\kk C$.
We write 
$X_C:= X \times_{\kk} \Spec C.$  Recall that a {\em $C$-point module (over $R$)} is a graded factor $M$ of $R_C$ so that $M_i$ is rank 1 projective for $i \geq 0$.  An {\em $\ell$-shifted $C$-point module (over $R$)} is a factor of $(R_C)_{\geq \ell}$ that is rank 1 projective in degree $\geq \ell$.  
A {\em truncated $\ell$-shifted $C$-point module of length $m$} is a factor module of $(R_C)_{\geq \ell}$ so that $M_i$ is rank 1 projective over $C$ for $\ell \leq i \leq \ell+m-1$ and $M_i = 0$ for $i \geq \ell +m$.  Since these modules depend on  a finite number of parameters, they are clearly parameterized up to isomorphism by a projective scheme.
For fixed $\ell\leq n$, we denote the $\ell$-shifted length $(n-\ell+1)$ point scheme of $R$ by ${}_{\ell}Y_n$.
A point in ${}_{\ell}Y_n$ gives a surjection $R_{\geq \ell} \to M$ (up to isomorphism), or equivalently a submodule of $R_{\geq \ell}$ with appropriate Hilbert series. 
Thus we say that ${}_{\ell}Y_n$ parameterizes {\em embedded} (shifted truncated) point modules.
  The map $M \mapsto M/M_n$ induces a morphism $\chi_n:  {}_{\ell}Y_n \to {}_{\ell}Y_{n-1}$. 

For later use, we explicitly construct a projective embedding of ${}_{\ell}Y_n$.  

\begin{proposition}\label{prop-likeATV}
(cf. \cite[section~3]{ATV1990})
 Let $R$ be a connected graded $\kk$-algebra generated in degree 1.
\begin{enumerate}
 \item For all $\ell\leq n\in \NN$, there is a closed immersion
\[ {}_{\ell}\Pi_{n}:  {}_{\ell}Y_n \to \PP((R_1^{\otimes \ell})^{\vee}) \times \PP(R_1^{\vee})^{\times (n-\ell)}.\]
\item Fix $\ell\leq n$ and let 
\[ \pi:  \PP((R_1^{\otimes \ell})^{\vee}) \times \PP(R_1^{\vee})^{\times( n-\ell)}\to \PP((R_1^{\otimes \ell})^{\vee}) \times \PP(R_1^{\vee})^{\times (n-\ell-1)}\]
be projection onto the first $n-\ell$ factors.  Then the diagram
\beq\label{embedding} \xymatrix{
 {}_{\ell}Y_n \ar[r]^(0.25){{}_{\ell}\Pi_{ n}} \ar[d]_{\chi_n} & \PP((R_1^{\otimes \ell})^{\vee}) \times \PP(R_1^{\vee})^{\times (n-\ell)} \ar[d]^{\pi} \\
{}_{\ell}Y_{n-1} \ar[r]_(0.25){{}_{\ell}\Pi_{n-1}} & \PP((R_1^{\otimes \ell})^{\vee}) \times \PP(R_1^{\vee})^{\times (n-\ell-1)}
}
\eeq
commutes.

\end{enumerate}
\end{proposition}
 \begin{proof}
  (1)
Let $T=T^{\bullet}(R_1)$ denote the tensor algebra on the finite-dimensional $\kk$-vector space $R_1$.  
We identify $T_1$ canonically with $R_1$ and $T_\ell$ with $R_1^{\otimes \ell}$.

Given an element $f\in T_{n}$, we get an $(\ell, 1, \dots, 1)$-form 
\begin{displaymath}
\widetilde{f}: T_\ell^\vee \times (T_1^\vee)^{\times (n-\ell)} \longrightarrow \kk
\end{displaymath}
by pairing with $f$.  The map is $\kk$-multilinear, hence $\widetilde{f}$ defines a hypersurface $Y(\widetilde{f})$ in
$\mathbb{P}(T_\ell^\vee)\times (\mathbb{P}(T_1^\vee)^{\times( n-\ell)}$.  More generally, given a collection $\{f_i\}$ of 
elements of $T_{n}$, we get a closed subscheme
\begin{displaymath}
Y(\{\widetilde{f_i}\})\subseteq \mathbb{P}(T_\ell^\vee)\times (\mathbb{P}(T_1^\vee)^{\times (n-\ell)}.
\end{displaymath}

Let $I$ be the kernel of the natural surjection $T \twoheadrightarrow R$.    Then the above construction gives a closed subscheme 
$Y(\widetilde{I}_{n})\subseteq \mathbb{P}(T_\ell^\vee)\times (\mathbb{P}(T_1^\vee)^{\times (n-\ell)}$.  We claim that $ Y(\widetilde{I}_{n})$ is naturally isomorphic to ${}_{\ell}Y_n$.

Let $C$ be a commutative $\kk$-algebra and let $R_C = R\otimes_{\kk} C$, $T_C = T\otimes_{\kk} C$ with the gradings induced from $R$ (respectively $T$).  Suppose that $\overline{\alpha}: (R_C)_{\geq \ell}\rightarrow M$ is an embedded $\ell$-shifted truncated $C$-point module of length $n-\ell+1$.  We write $\alpha: (T_C)_{\geq \ell} \rightarrow M$ for the composite of the two surjections.  Assume that $M = \oplus_{i=\ell}^{n} m_i\cdot C $ is a free graded $C$-module on generators $m_i$.  Then $\alpha$ determines $C$-linear maps $a_j:  T_1 \otimes_{\kk} C\rightarrow C$ 
for $1\leq j\leq n-\ell$ by $m_{\ell+j-1} x = m_{\ell+j} a_j(x)$ for $x\in T_1 \otimes_{\kk} C$, and a $C$-linear map
$b:  T_\ell \otimes_{\kk} C\rightarrow C$ by $\alpha(y) = m_\ell b(y)$ for $y\in  T_\ell \otimes_{\kk} C$.  Since $M$ is a (shifted, truncated) point module, hence generated in degree $\ell$, these maps are surjective.  Hence they determine a morphism
\begin{displaymath}
\Pi(\alpha) = (b, a_1,\dots, a_{n-\ell}): \operatorname{Spec}(C) \longrightarrow \mathbb{P}(T_\ell^\vee)\times (\mathbb{P}(T_1^\vee)^{\times (n-\ell)}.
\end{displaymath}
We see immediately from the construction that if 
$f\in I_{n}\otimes_{\kk} C$, then $\widetilde{f}(\Pi(\alpha)) = 0$.  In particular, $\Pi(\alpha)$ factors through
$Y(\widetilde{I}_{n})$.  

It follows immediately that the above construction defines a morphism $\Pi$ from the moduli functor of shifted truncated point modules with {\em free} (as $C$-modules) graded components to $Y(\widetilde{I}_{n})$.  Since the latter is a scheme, hence a sheaf in the fpqc topology, $\Pi$ induces a morphism, which we denote by ${}_{\ell}\Pi_n$, from the 
moduli functor ${}_{\ell}Y_{n}$ for all shifted truncated $C$-point modules over $R$ to $Y(\widetilde{I}_{n})$.

\begin{claim}
The morphism ${}_{\ell}\Pi_n$ is an isomorphism: that is, $Y(\widetilde{I}_{n}) \cong {}_{\ell}Y_n$ represents the moduli functor of embedded truncated 
$\ell$-shifted  $C$-point modules over $R$ of length $n-\ell+1$.
\end{claim}
\begin{proof}[Proof of Claim.]
A morphism $\operatorname{Spec}(C)\rightarrow Y(\widetilde{I}_{n}) \subseteq \mathbb{P}(T_\ell^\vee)\times (\mathbb{P}(T_1^\vee)^{\times (n-\ell)}$  gives a tuple 
$(b, a_1, \dots, a_{n-\ell})$ where each $a_j$ is a surjective $C$-linear map $a_j:  T_1 \otimes_{\kk} C\rightarrow N_j$ and each $N_j$ is a finitely generated projective $C$-module of rank $1$; and $b: T_\ell \otimes_{\kk}  C\rightarrow M_\ell$ is a surjective $C$-linear map onto a finitely generated projective $C$-module $M_\ell$ of rank $1$. 

Assume first that $M_\ell$ and each $N_j$  is a free $C$-module, and choose basis elements.  Define a $T_C$-module $M = \oplus_{j=\ell}^{n}  m_j  \cdot C$ by $m_{j-1} x = m_j a_{j-\ell}(x)$ for $x\in T_1 \otimes_{\kk} C$.  
Moreover, define a map $\alpha: (T_C)_{\geq \ell} \rightarrow M$ by $\alpha(y) = m_\ell b(y)$ for $y\in T_\ell$ and extending linearly.  
It is a consequence of the construction of $Y(\widetilde{I}_{n})$ that 
the map $\alpha$ factors through $(R_C)_{\geq \ell}$ and makes $M$ an $\ell$-shifted truncated $C$-point module over $R$.  

Next, we observe that the functor $\Pi$  (on shifted truncated point modules with free $C$-module components) and the above construction (on maps $\operatorname{Spec}(C)\rightarrow Y(\widetilde{I}_{n})$ for which the modules $N_j$ and $M_\ell$ are free $C$-modules) give mutual inverses.  This follows from the argument of \cite[3.9]{ATV1990}, which uses only the freeness condition.  In particular, the functor $\Pi$ 
is injective.  

To prove that the sheafification ${}_{\ell}\Pi_n$ is an isomorphism, then, it suffices to show that $\Pi$ is
locally surjective: that is, for every morphism $\operatorname{Spec}(C)\rightarrow Y(\widetilde{I}_{n})$, there is a faithfully flat morphism $\operatorname{Spec}(C')\rightarrow \operatorname{Spec}(C)$, and a shifted truncated $C'$-point module with free $C'$-module components, whose image under $\Pi$ is the composite map 
$\operatorname{Spec}(C')\rightarrow Y(\widetilde{I}_{n})$.  But it is standard that such a homomorphism $C\rightarrow C'$ can be found that makes each $N_j$ and $M_\ell$ trivial, and now the construction of the previous paragraph proves the existence of the desired shifted truncated point module.  This completes the proof of the claim.
\end{proof}
\noindent
Part (1) follows from the claim. 
(2) follows by construction.
\end{proof}

\begin{proposition}\label{prop:ATV stabilise}
(cf. \cite[Proposition~3.6]{ATV1990})
 Let $R$ be a connected graded $\kk$-algebra generated in degree one.  Let $n > \ell$ and consider the truncation morphism
\[ \chi_n:  {}_{\ell}Y_n \to {}_{\ell}Y_{n-1}.\]
Let $y \in {}_{\ell}Y_{n-1}$ and suppose that $\dim \chi_n^{-1}(y) = 0$.  Then $\chi_n^{-1}$ is defined and is a local isomorphism locally in a neighborhood of $y$. 
\end{proposition}
\begin{proof}
We consider the commutative diagram \eqref{embedding} of Proposition~\ref{prop-likeATV}(2).  
By Proposition~\ref{prop-likeATV}(1) the horizontal maps are closed immersions.
 Since the defining equations of $Y(\widetilde{I}_{n}) \subseteq \mathbb{P}(T_\ell^\vee)\times (\mathbb{P}(T_1^\vee)^{\times(n-\ell)}$ are $(\ell, 1, \dots, 1)$-forms and in particular are linear in the last  coordinate, the fibers of $\chi_n$ are linear subspaces of $\PP(T_1^{\vee})$.  The result  follows as in the proof of \cite[Proposition~3.6(ii)]{ATV1990}.
\end{proof}

\begin{proposition}\label{prop-likeRS}
 {\em (cf. \cite[Proposition~2.5]{RS})}
Let $R$ be a noetherian connected graded $\kk$-algebra generated in degree 1.  
For $n >\ell \geq 0$, define $\chi_n:  {}_{\ell}Y_n \to {}_{\ell}Y_{n-1}$ as in the beginning of the section.  
Let $n_0 \geq 0$ and let  $\{ y_n \in {}_{\ell} Y_n \st n \geq n_0\}$ be a sequence of (not necessarily closed) points so that $\chi_n(y_n) = y_{n-1}$ for all $n > n_0$.  
Then for all $n \gg n_0$ the fibre $\chi_n^{-1}(y_{n-1})$ is a singleton and $\chi_n^{-1}$ is defined and is a local isomorphism at $y_{n-1}$.
\end{proposition}
\begin{proof}
 This follows as in the proof of \cite[Proposition~2.5]{RS}, using Proposition~\ref{prop:ATV stabilise} instead of \cite[Proposition~3.6(ii)]{ATV1990}.
\end{proof}

We are interested in studying the limit ${}_{\ell}\Yt := \varprojlim {}_{\ell}Y_n$; however, we first  study the point schemes of $\sh{S}$.  That is, for $n \geq \ell \geq 0$,   
it is clear that we may also define a scheme that  parameterizes factor modules $\sh{M}$ of $\sh{S}_{\geq \ell}$ so that, as graded $\struct_X$-modules, $\sh{M} \cong \kk_x t^{\ell} \oplus \cdots \oplus \kk_x t^n$ for some $x \in X$.  
We say that $x$ is the {\em support} of $\sh{M}$.  
We denote this $\ell$-shifted length $(n-\ell +1)$ truncated point scheme of $\sh{S}$ by ${}_{\ell}Z_n$.    
More formally, a $\operatorname{Spec}(C)$-point of ${}_{\ell}Z_n$ will be a factor module $\sh{M}$ of $ \sh{S}_{\geq \ell} \otimes_{\kk}  C$ which is isomorphic as a graded $\struct_{X_C}$-module 
to a direct sum $P_\ell \oplus \dots \oplus P_n$, where each $P_i$ is a coherent $\struct_{X_C}$-module that is finite over $C$ (in the sense that its support in $X_C$ is finite over $\operatorname{Spec}(C)$) and is a rank one projective $C$-module (which is well defined because of the finite support condition). 
We let  $Z_n:={}_0Z_n$ be the unshifted length $n+1$ point scheme of $\sS$.

For all $n  > \ell \geq 0$, there are maps
\[
\phi_n:  {}_{\ell}Z_n  \to {}_{\ell}Z_{n-1} \;\;\; \text{defined by} \;\;\; 
\sh{M}  \mapsto \sh{M}/\sh{M}_n.\]

If $\ell = 0$ and $\sh{M}$ is a truncated point module of length $n$ over $\sh{S}$, then $\sh{M}[1]_{\geq 0}$ is also cyclic (since $\sh{S}$ is generated in degree 1) and so is a factor of $\sh{S}$, in a unique way up to scalar. This induces a map 
\[
\psi_n:  Z_n  \to Z_{n-1} \;\;\; \text{defined by} \;\;\; 
\sh{M}  \mapsto \sh{M}[1]_{\geq 0}.
\]
It is clear that $\psi_n$ and $\phi_n$ map relevant components to relevant components.

\begin{lemma}\label{lem-faces}
Let $f_n: {}_{\ell}Z_n \to X$ be the map that sends a module $\sh{M}$ to its support.   The diagrams
\beq\label{d1}
\xymatrix{
& {}_{\ell}Z_n \ar[dd]_{\phi} \ar[ld]_{f} \\
X & \\
&  {}_{\ell}Z_{n-1} \ar[lu]^{f} 
}
\eeq
and
\beq\label{d2}
\xymatrix{
{}_0Z_n \ar[d]_{\psi} \ar[rr]^{f = \phi^n} &&  X \ar[d]^{\sigma} \\
{}_0Z_{n-1} \ar[rr]_{f = \phi^{n-1}} && X
}
\eeq
commute.
\end{lemma}

\begin{proof}
It is clear by construction that if $\sh{M}$ is a shifted truncated point module and $\sh{M}'$ is a further factor of $\sh{M}$, then $\sh{M}$ and $\sh{M}'$ have the same support.  Thus \eqref{d1} commutes.

Let $\sh{M}$ be a truncated point module corresponding to a point $z \in {}_0Z_n$.  
Let $x := f(z)$.  
By \cite[Lemma~5.5]{KRS}, we have
\[ \sh{M}[1]_n \cong (\sh{M}_{n+1})^{\sigma^{-1}} \cong (\kk_x)^{\sigma^{-1}} \cong \kk_{\sigma(x)}.\]
Thus $f \psi(z) = \sigma f(z)$, as claimed, and \eqref{d2} commutes.
\end{proof}

Recall that $\sS_n = \sI_n \sL_n$.  

\begin{proposition}\label{prop-Z}
For $n\geq 0$, let  $X_n$ be the blowup of $X$ at  $\sh{I}_n$, and let $\alpha_n, \beta_n:  X_n \to X_{n-1}$ be as in Corollary~\ref{cor-X-maps}.    

Then for all $n > \ell \geq 0$ there are isomorphisms $j_n:  X_n \to {}_{\ell}Z_n^o \subseteq {}_{\ell}Z_n$ so that the diagrams
\[ \xymatrix{
X_n \ar[d]^{\alpha_n} \ar[r]^{j_n} & {}_{\ell}Z_n \ar[d]^{\phi_n}  \\
X_{n-1} \ar[r]_{j_{n-1}} & {}_{\ell}Z_{n-1}
}
 \quad\quad \mbox{and} \quad\quad
 \xymatrix{
X_n \ar[d]_{\beta_n} \ar[r]^{j_n} &
{}_0 Z_n  \ar[d]_{\psi_n} \\
X_{n-1} \ar[r]_{j_{n-1}} & {}_0 Z_{n-1}
}
\]
commute.  
\end{proposition}

\begin{proof}
We will do the case that $\ell=0$; the general case is similar.  Let $Z_n := {}_{0} Z_n$.  For $0 \leq i \leq n$, let $W_i = \Proj \shSym_X \sI_i$ be the scheme parameterizing colength 1 ideals inside $\sh{I}_i$, and let $c_i:  X_i \to W_i$ be the map from Lemma~\ref{lem-ideals}.  
Let 
\[ r_n: X_n \to X \times W_1 \times \cdots \times W_n\]
be the composition
\[ \xymatrix{
X_n \ar[rrr]^(.3){\alpha^n \times \alpha^{n-1} \times \cdots \times 1} &&&
X \times X_1 \times \cdots \times X_n 
\ar[rr]^{c_0 \times \cdots \times c_n} && X \times W_1 \cdots \times W_n. }\]
Since this is the composition of the graph of a morphism with a closed immersion, it is also a closed immersion and is an isomorphism onto its image.

Now, a point in $Z_n$ corresponds to an ideal $\sh{J} \subset \sh{S}$ so that the factor is a truncated point module of length $n+1$, and there is thus a canonical closed immersion
$\delta_n:  Z_n \to X \times W_1 \times \cdots W_n.$
The map $\delta_n$ sends a graded  right ideal $\sh{J}$ of $\sh{S}$  to the tuple $(\sh{J}_0, \sh{J}_1, \ldots, \sh{J}_n)$.  

Conversely, a point $(\sh{J}_0, \ldots, \sh{J}_n) \in X \times W_1 \times \cdots \times W_n$ is in $\im(\delta)$ if and only if we have $\sh{J}_i \sh{S}_j^{\sigma^i} \subseteq \sh{J}_{i+j}$ for all $i+j \leq n$.  It follows from Lemma~\ref{lem-blowupmaps}(c) that $\im(r_n) \subseteq \im(\delta_n)$.  Since $r_n$ and $\delta_n$ are closed immersions and $X_n$ is reduced, we may define
$
j_n = \delta_n^{-1} r_n:  X_n \longrightarrow Z_n.
$

Let $U := X \smallsetminus \Cosupp \sh{I}_n$. Then $f^{-1}_n$ and $a^{-1}_n$ are defined on $U$, and the diagram
\[ \xymatrix{
& X \times W_1 \times \cdots \times W_n & \\
X_n \ar[rr]_{j_n} \ar[ru]^{r_n} && Z_n \ar[lu]_{\delta_n} \\
& U \ar[lu]^{a_n^{-1}} \ar[ru]_{f_n^{-1}} & 
}\]
commutes.  
Since $r_n$ and $\delta_n$ are closed, 
\[ r_n(X_n) = r_n(\bbar{a_n^{-1}(U)}) = \bbar{r_na_n^{-1}(U)} = \bbar{\delta_n f_n^{-1}(U)}
= \delta_n(\bbar{f_n^{-1}(U)}) = \delta_n(Z_n^o).\]  
Therefore, $j_n$ is an isomorphism to $Z_n^o$.

Let $q:  X \times W_1 \times \cdots \times W_n \to X \times W_1 \times \cdots \times W_{n-1}$ be projection onto the first $n$ factors.  
Consider the diagram
\[
\xymatrix{
X_n \ar[r]^(0.3){r_n} \ar[d]_{\alpha_n} & X\times W_1 \times \cdots \times W_n \ar[d]_{q} &
Z_n \ar[l]_(0.3){\delta_n} \ar[d]_{\phi_n} \\
X_{n-1} \ar[r]_(0.3){r_{n-1}}  & X\times W_1 \times \cdots \times W_{n-1}  &
Z_{n-1} \ar[l]^(0.3){\delta_{n-1}}  .
} \]
From the definitions of $r_n$ and $\delta_n$ we see that this diagram commutes; since $j_n = \delta_n^{-1} r_n$, the diagram
\[
\xymatrix{
X_n \ar[r]^{j_n} \ar[d]_{\alpha_n} & Z_n \ar[d]^{\phi_n} \\
X_{n-1} \ar[r]_{j_{n-1}} & Z_{n-1}
}\]
commutes.

Let $X'_n := \im(r_n) \subset X \times W_1 \times \cdots \times W_n$, and let $\beta'_n:  X'_n \to X'_{n-1}$ be the map induced from $\beta_n$.  The proof of Corollary~\ref{cor-X-maps} shows that if
$\bigl(\sJ_0, \sJ_1, \ldots, \sJ_n\bigr) \in X'_n,$
then its image under $\beta'_n$ is
\beq \label{ddag}
 \beta'_n\Bigr( \bigl(\sJ_0, \sJ_1, \ldots, \sJ_n\bigr)  \Bigl) = 
 (\sF_0, \ldots, \sF_{n-1}) \in X'_{n-1},
 \eeq
 where $\sF_i := (\sJ_{i+1}: \sI_1)^{\sigma^{-1}} \cap \sI_i$.

 Now let $\sJ$ be the ideal defining a truncated point module of length $n+1$.  By abuse of notation, we think of $\sJ$ as a point in $Z_n$.  Let $\sM:= \sS/\sJ$.  Then $\psi(\sJ)_i = (\Ann_{\sS}(\sM_1))_i$, for $0 \leq i \leq n-1$.
 Thus
 $\sI_1 ( \psi(\sJ)_i)^{\sigma} \subseteq \sJ_{i+1}$
 or 
 $\psi(\sJ)_i \subseteq \bigl( \sJ_{i+1}: \sI_1 \bigr)^{\sigma^{-1}} \cap \sI_i.$
 If $\sJ \in \im(j_n)$, then we have equality by the computation in \eqref{ddag}.  Thus the diagram
 \[ 
 \xymatrix{
 X_n \ar[r]_{r_n} \ar[d]_{\beta_n} \ar@/^10pt/[rr]^{j_n} & X_n' \ar[r]_{\delta_n^{-1}} \ar[d]_{\beta'_n} & Z_n \ar[d]^{\psi_n} \\
 X_{n-1} \ar[r]^{r_{n-1}} \ar@/_10pt/[rr]_{j_{n-1}} & X_{n-1}' \ar[r]^{\delta_{n-1}^{-1}} & Z_{n-1}
 }\]
 commutes.
\end{proof}

To end this section, we construct stacks ${}_{\ell}\Zt$ and ${}_{\ell}\Yt$ that are fine moduli spaces for (shifted) embedded point modules and give some of their properties.  
 A version of the following result was known long ago to M. Artin; however, it  does it seem to have appeared in the literature.

\begin{theorem}\label{fine moduli}
Fix $\ell \in \NN$.  
For $n \geq \ell$, let $X_n$ be the  blowup of $X$ at $\sh{I}_n$.
Let
${}_{\ell}Y_n$ be the moduli space of  $\ell$-shifted length $(n - \ell +1)$  point modules over $S$.
Let ${}_{\ell}Z_n$ be the moduli space of  $\ell$-shifted length $(n - \ell +1)$  point modules over $\sh{S}$.
Define the morphisms
$\chi_n:  {}_{\ell}Y_n   \to {}_{\ell}Y_{n-1}$, $\phi_n:  {}_{\ell}Z_n \to {}_{\ell}Z_{n-1}$, and $\alpha_n:  X_n \to X_{n-1}$ as above.
 Let
\[ {}_{\ell}\Zt := \varprojlim_{\phi_n} {}_{\ell}Z_n, \;\;\;\;\;\; \;\;\;\; 
{}_{\ell}\Yt := \varprojlim_{\chi_n} {}_{\ell}Y_n, \;\;\; \text{and} \;\;\; 
\Xt := \varprojlim_{\alpha_n} X_n.\]
Then the stack ${}_{\ell}\Yt $ is a sheaf in the fpqc topology and  is a fine moduli space for $\ell$-shifted embedded point modules over $S$.  
The stack  $ {}_{\ell} \Zt$ is noetherian fpqc-algebraic and is a fine moduli space for $\ell$-shifted embedded point modules over $\sS$.
The  relevant component of ${}_{\ell} \Zt$ is isomorphic to $\Xt$.
\end{theorem}

\begin{proof}
We suppress the subscript $\ell$ in the proof.  

For $n \geq \ell$, let $F_n$ be the moduli functor for  truncated $\ell$-shifted  point modules over $S$, so $ Y_n \cong F_n$. 
 Define a contravariant functor 
\[ F:  \text{Affine schemes} \longrightarrow \text{Sets} \]
\[ \Spec C \mapsto \big\{ \text{Embedded $\ell$-shifted $C$-point modules over $S$} \big\}.\]
By descent theory, $F$ is a sheaf in the fpqc topology.  More precisely, recall that quasicoherent sheaves form a stack in the fpqc 
topology (see \cite[Section~4.2.2]{Vistoli}); consequently the (graded) quotients of $S_{\geq \ell}$ form a sheaf of sets in the 
fpqc topology.  Moreover, as in the first paragraph of Section~4.2.3 of \cite{Vistoli}, those quotients of $S_{\geq \ell}$ that are $S$-module quotients form
 a subsheaf in the fpqc topology; this subsheaf is $F$.  
It is formal that  $F$ is isomorphic to the functor $h_{\Yt}$.

%

Likewise, $ \Zt$ parameterizes $\ell$-shifted point modules over $\sh{S}$.  
We show that  (i)--(v) of Proposition~\ref{prop-algspace} apply to $\Zt$.
We have $\sS_n = \sI_n \sL_n$; let $P_n \subset X$ be the subscheme defined by $\sI_n$.
Consider the maps
$ \xymatrix{Z_n \ar[r]^{\phi_n} \ar@/_/[rr]_{f_n} & Z_{n-1} \ar[r]^{f_{n-1}} & X}$
from Lemma~\ref{lem-faces}.  
Now, $f_n^{-1}$ is defined away from $P_n$, and $\cup_n P_n$ is a countable critically dense set.  
Thus (i), (ii) hold.

Let $x \in \bigcup P_n$.  As the points in $P$ have infinite orbits, there is some $m \in \NN$ so that $x \not \in \sigma^{-n}(P)$ for all $n \geq m$.  
 Let $z_m \in Z_m$ with $f_m(z_m) = x$, corresponding to a right ideal $\sJ \subseteq \sS_{\geq \ell}$ with $S_{\geq \ell}/\sJ \cong \bigoplus_{j = \ell}^m \sO_x$.  
Let $\sJ' := \sJ_{\leq m} \cdot \sS$.
For any $j \geq 0$ we have $(\sS_j^{\sigma^m})_x = (\sL_j^{\sigma^m})_x$, and so  $\sJ'$ gives the unique preimage of $z_m$ in $\Zt$.   
A similar uniqueness holds  upon base extension, so the scheme-theoretic preimage of $z_m$ in $\Zt$ is  a $\kk$-point, and $\Phi_m^{-1}$ is defined and is a local isomorphism at $z_m$.  

For $j \geq \ell$, let $W_j := \Proj \shSym_X \sI_n$.  
As in the proof of Proposition~\ref{prop-Z}, we may regard $Z_n$ as a closed subscheme of $W_\ell \times \cdots \times W_n$, and (iv) and (v)  follow from this and the fact that the orbits of points in $P$ are infinite. 
By Proposition~\ref{prop-algspace}, then,  $\Zt$ is a noetherian fpqc-algebraic stack.

Consider the morphisms $j_n: X_n \stackrel{\cong}{\to} Z_n^o \subseteq Z_n$ from Proposition~\ref{prop-Z}.  Commutativity of the first diagram in that proposition gives an induced isomorphism $  j:  \Xt \to  \Zt^o \subseteq \Zt$.  
\end{proof}

\section{Comparing moduli of points}\label{COMPARE}

In this section we prove that ${}_{\ell}\Yt$ is also  noetherian fpqc-algebraic, and that, 
at least for sufficiently large $\ell$, the stacks ${}_{\ell} \Zt$ and ${}_{\ell}\Yt$ are isomorphic.

In the following pages, we will always use the following notation.  We write a commutative $\kk$-algebra $C$ as $p: \kk \to C$, to indicate the structure map explicitly.  
We write $X_C := X \otimes_\kk \Spec C$.   
We abuse notation and let the projection  map 
$1 \otimes p:  X_C \to X$ 
also be denoted by $p$.
We let $q: X_C \to \Spec C$ be projection on the second factor.

Suppose that $p:  \kk \to C$ is a commutative $\kk$-algebra, and $y:  \Spec C \to X$ is a $C$-point of $X$.  Then $y$ determines a section of $q$, which we also call $y$.  This is a morphism $y:  \Spec C \to X_C$.  We define $\sI_y \subseteq \sO_{X_C} $ to be the ideal sheaf of the corresponding closed subscheme of $X_C$.  We define $\sO_y := \sO_{X_C}/\sI_y$.

We use the relative regularity results from Section~\ref{BACKGROUND} to study the pullbacks of the sheaves $\sS_n$ to $X_C$.  Fix a very ample invertible sheaf $\sO_X(1)$ on $X$, which we will use to measure regularity.

\begin{lemma}\label{lem-ample}
 Let $p:  \kk \to C$ be a commutative  noetherian $\kk$-algebra.
Then $\{p^* \sS_n\}_{n \geq 0}$ is a right ample sequence on $X_C$.
\end{lemma}
\begin{proof}
 Let $\sF$ be a coherent sheaf on $X_C$.  
By \cite[Corollary~3.14]{RS-0}, $\lim_{n \to \infty} \reg(\sS_n) = -\infty$.  
Thus $\lim_{n \to \infty} \reg(p^* \sS_n) = -\infty$ by Lemma~\ref{lem-CM1}. 
Since each $ \sS_n$ is invertible away from a dimension 0 set, $p^* \sS_n$ is invertible away from a locus of relative dimension 0.  By Proposition~\ref{prop-Dennis} $\sF \otimes_{X_C} p^* \sS_n$ is 0-regular for $ n \gg 0$.
Theorem~\ref{thm-Mumford} shows that $(1)$, $(2)$ of Definition~\ref{def-ample} apply.
\end{proof}

We  now prove a uniform regularity result for certain subsheaves of a pullback of  some $\sS_n$.  
\begin{lemma}\label{lem-posparam-correct}
 There exists $m \geq 0$ so that the following holds for any $n \geq m$:
for any  commutative noetherian $\kk$-algebra $p:  \kk \to C$, 
for any $C$-point $y$ of $X$, 
and for any coherent sheaf $\sK$ on $X_C$ so that  $\sI_y p^* \sS_n \subseteq \sK \subseteq p^* \sS_n$, 
then $\sK$ is 0-regular.  In particular, $\sK$ is globally generated and $R^1q_* \sK= 0$.
\end{lemma}
\begin{proof}
 Let $D$ be the constant from Proposition~\ref{prop-Dennis}, and let $r:= \reg(\sO_X)$.  
By \cite[Corollary~3.14]{RS-0}, we have $\lim_{n \to \infty} \reg(\sS_n) = - \infty$. 
 Let $m$ be such that for all $n \geq m$, $\reg(\sS_n) \leq -r -D - 1$. 
  We claim this $m$ satisfies the conclusions of the lemma.

Fix a commutative  noetherian $\kk$-algebra $p: \kk \to C$ and a $C$-point $y$ of $X$.  
 We first claim that $ \reg(\sI_y) \leq r + 1$.
To see this, let $i \geq 1$ and consider the exact sequence
\[ 
 R^{i-1} q_* \sO_{X_C}(r+1-i) \stackrel{\alpha}{\to}R^{i-1}q_* \sO_y(r+1-i) \to R^iq_* \sI_y(r+1-i) \to R^iq_* \sO_{X_C}(r+1-i). 
\]
The last term vanishes, as $\sO_{X_C}$ is $(r+1)$-regular by Lemma~\ref{lem-CM1} and Theorem~\ref{thm-Mumford}(1).  
If $i \geq 2$, then $R^{i-1}q_* \sO_y(r+1-i)=0$ for dimension reasons, so $R^iq_* \sI_y(r+1-i) =0$.  
On the other hand, if $i =1$, then because $\sO_{X_C}(r)$ is 0-regular,  by Lemma~\ref{SL1} $\alpha$ is surjective.
Again, $R^iq_*\sI_y(r+1-i) = 0$.  Thus $\sI_y$ is $(r+1)$-regular as claimed.

Let $n \geq m$.  
By  Lemma~\ref{lem-CM1},  $\reg (p^* \sS_n) = \reg (\sS_n)$.  
Note that $\sI_y$ and $p^* \sS_n$ are both locally free away from a set of relative dimension 0.  
Thus the hypotheses of  Proposition~\ref{prop-Dennis} apply, and by that result we have
\[
 \reg(\sI_y \otimes_{X_C} p^* \sS_n) \leq \reg(\sI_y) + \reg (p^* \sS_n) + D \leq r+1 + D+ \reg (p^* \sS_n) 
=r+1+D + \reg (\sS_n).
\]
Our choice of $n$ ensures this is non-positive.  
In particular, $\sI_y \otimes_{X_C} p^* \sS_n$ is 0-regular.

Let $\sI_y p^* \sS_n \subseteq \sK \subseteq p^* \sS_n$.  
There is a natural map $f:  \sI_y \otimes_{X_C} p^* \sS_n \to \sK$ given by the  composition $\sI_y \otimes_{X_C} p^* \sS_n \to \sI_y \cdot p^* \sS_n \subseteq \sK$.   
The kernel and cokernel of $f$ are supported on a set of relative dimension 0, and it is an easy exercise that $\sK$ is therefore also 0-regular.  
By Theorem~\ref{thm-Mumford}, $\sK$ is globally generated and $R^1q_* \sK = 0$, as claimed.
\end{proof}

\begin{defn}\label{positivity parameter}
We call a positive integer $m$ satisfying the conclusion of 
Lemma~\ref{lem-posparam-correct}  a 
{\em \pp}.
\end{defn}

The proof of 
Lemma~\ref{lem-posparam-correct} shows that if we are willing to replace $\sL$ by a sufficiently ample invertible sheaf, we may in fact assume that $m=1$ is a positivity parameter.  (By \cite[Theorem~1.2]{Keeler2000}, the existence of a $\sigma$-ample sheaf means that any ample invertible sheaf is $\sigma$-ample.)

\begin{corollary}\label{cor-ptideal}
Let $p: \kk \to C$ be a noetherian commutative $\kk$-algebra.  
  Let $m$ be a \pp\ (Definition \ref{positivity parameter}) and let $n \geq m$.

\begin{enumerate}
\item If $\sh{J} \subset p^* \sh{S}_n$ is a sheaf on $X_C$ so that $p^* \sh{S}_n/\sh{J}$ has support on $X_C$ that is finite over $\operatorname{Spec}(C)$ and
is a rank 1 projective $C$-module, then $q_* \sh{J}$ is a $C$-submodule of $q_* p^*\sS_n =S_n \otimes C$ such that the cokernel is rank 1 projective.

\item If $\sK \subsetneqq \sJ \subset p^* \sS_n$ are sheaves on $X_C$ so that $p^* \sS_n/\sJ$ is a rank 1 projective $C$-module, then $q_* \sK \subsetneqq q_* \sJ$.

\item If $\sJ, \sJ' \subseteq p^*\sS_n$ are sheaves on $X_C$ so that $p^* \sS_n/\sJ$ and $p^*\sS_n/\sJ'$ are rank 1 projective $C$-modules, then
$q_* \sJ = q_* \sJ'$ if and only if $\sJ = \sJ'$.
\end{enumerate}
\end{corollary}

\begin{proof}
 (1)  Let $x \in \Spec C$ be a closed point.  
Consider the fiber square
\[ \xymatrix{
 X_x \ar[r] \ar[d] \ar@{}[rd]|{\square} & X_C\ar[d]^{q} \\
\{x\} \ar[r] & \Spec C.}
\]
 Let $\sJ_x := \sJ|_{X_x }$. 
 Since $p^* \sS_n/\sJ$ is flat over $\Spec C$, we have $\sJ_x \subseteq p^* \sS_n |_{X_x} \cong \sS_n\otimes_{\kk} \kk(x)$.  
Further, $(\sS_n\otimes_{\kk} \kk(x))/\sJ_x \cong \sO_{x}$.  
By our choice of $n$, therefore, $H^1(X_x, \sJ_x) = 0$.

Now $\sJ$ is the kernel of a surjective morphism of flat sheaves and so is flat over $\Spec C$. 
  Since  $H^1(X_x, \sJ_x) = 0$, 
 by the theorem on cohomology and base change \cite[Theorem~III.12.11(a)]{Hartshorne}
we have $R^1 q_* \sJ \otimes_C \kk(x) = 0$.  
The $C$-module $R^1q_* \sh{J}$ thus vanishes at every closed point and is therefore 0.  

The complex 
\[ 0 \to q_*\sh{J} \to q_* p^* \sS_n \to q_*(p^* \sh{S}_n/\sh{J}) \to 0\]
is thus exact.  By assumption (1),  $q_*( p^* \sh{S}_n /\sh{J})$ is a rank 1 projective $C$-module.
Since cohomology commutes with flat base change \cite[Proposition~III.9.3]{Hartshorne}, we have  $q_* p^* \sS_n \cong H^0(X, \sS_n)\otimes_{\kk} C = S_n \otimes_{\kk} C$.

(2). 
Since $m$ is a positivity parameter, $\sJ$ is globally generated, and it follows immediately that $q_* \sK \neq q_* \sJ$.  

(3).  From (2) we have 
\[ q_*(\sJ\cap \sJ') = q_* \sJ \iff \sJ \cap \sJ' = \sJ \iff \sJ \subseteq \sJ'.\]
It follows from our assumptions that this occurs if and only if $\sJ = \sJ'$.  Further,  
\[ q_* \sJ = q_*(\sJ \cap \sJ') = q_* (\sJ) \cap q_* \sJ' \iff q_* \sJ \subseteq q_* \sJ'.\]
From (1) we obtain that this is equivalent to $q_* \sJ = q_* \sJ'$.  
\end{proof}

We now apply these regularity results to show that ${}_{\ell}\Yt \cong {}_{\ell} \Zt$ for $\ell \gg 0$.  
We will need the following easy lemma:

\begin{lemma}\label{lem:s is scheme-theoretically surjective}
 Let $A, B$ be commutative noetherian local $\kk$-algebras with residue field $\kk$.  Let $s: A \to B$ be a local homomorphism.
If $s^*:  \Hom^{\alg}(B,C) \to \Hom^{\alg}(A,C)$ is surjective for all finite-dimensional commutative local $\kk$-algebras $C$, then $s$ is injective.  
\end{lemma}
\begin{proof}
 Let $\mf{m}$ be the maximal ideal of $A$ and let $\mf{n}$ be the maximal ideal of $B$.
Let $f \in \ker s$.  
Suppose first that there is some $k$ so that $f \in \mf{m}^{k-1} \ssm \mf{m}^k$.
Let $C := A/\mf{m}^k$ and let $\pi: A \to C$ be the natural map. 
Then $C$ is a finite-dimensional artinian local $\kk$-algebra.  
Now, $\pi(f) \neq 0$ but $s(f)=0$.  
Thus $\pi \not \in \im(s^*)$, a contradiction.

We thus have $\ker s \subseteq \bigcap_k \mf{m}^k$.  
By the Artin-Rees lemma, $\ker s = 0$.
\end{proof}

\begin{proposition}\label{prop:immersion of point schemes}
Let $m$ be a \pp, and let $n \geq \ell \geq m$.  Let
${}_{\ell}Z_n$ be the $\ell$-shifted length $(n-\ell+1)$ point scheme of $\sh{S}$, with truncation morphism $\phi_n:  Z_n \to Z_{n-1}$ as in Proposition~\ref{prop-Z}.  Let 
${}_{\ell}Y_n$ be the   $\ell$-shifted  length $(n-\ell+1)$ point scheme of $S$, with truncation morphism
$ 
\chi_n:  {}_{\ell}Y_n   \to {}_{\ell}Y_{n-1}$ as in Theorem~\ref{fine moduli}.  
Let ${}_{\ell}\Zt := \varprojlim {}_{\ell} Z_n$ and let ${}_{\ell}\Yt :=  \varprojlim Y_n$.

Then the global section functor induces a closed immersion
 $s_n :  {}_{\ell}Z_n \to {}_{\ell}Y_n$
   so that the diagram
\beq\label{snow} \xymatrix{
{}_{\ell}Z_{n}\ar[r]^{s_n} \ar[d]_{\phi_{n}}	& {}_{\ell}Y_{n} \ar[d]^{\chi_{n}} \\
{}_{\ell}Z_{n-1} \ar[r]_{s_{n-1}}	& {}_{\ell}Y_{n-1}
} \eeq
commutes. 
\end{proposition}

\begin{proof}
 Let $p:  \kk \to C$ be a commutative noetherian $\kk$-algebra. 
Let $p:  X_C \to X$ and $q:  X_C \to \Spec C$ be the two projection maps, as usual.
Note that if $\sh{J} \subset p^* \sh{S}_{\geq \ell}$ is the defining ideal of an $\ell$-shifted length $(n-\ell+1)$ truncated $C$-point module over $\sh{S}$, then by Corollary~\ref{cor-ptideal}(1), $q_*\sh{J}$ is the defining ideal of an $\ell$-shifted length $(n-\ell+1)$ truncated $C$-point module over $S$. 
 Thus  $s_n$ is well-defined. 
By Corollary~\ref{cor-ptideal}(3), $s_n$ is injective on $\kk$-points and on $\kk[\epsilon]$ points.  It is standard (cf. the proof of \cite[Theorem~14.9]{Harris}) that, because $s_n$ is projective, $s_n$ is a closed immersion.  
That \eqref{snow} commutes is immediate.
\end{proof}

We will see that ${}_{\ell}\Yt$ and ${}_{\ell}\Zt$ are isomorphic.  It does not seem to be generally true that ${}_{\ell}Y_n$ and ${}_{\ell}Z_n$ are isomorphic; but we will see that there is an induced isomorphism of  certain naturally defined  closed subschemes.

Let $\ell \in \NN$.  For any $n$, let $\Phi_n: {}_{\ell}\Zt \to {}_{\ell}Z_n$ and $\Upsilon_n:  {}_{\ell}\Yt \to {}_{\ell}Y_n$ be the induced maps.  
For any $n \geq \ell$,  define ${}_{\ell}Z_n' \subseteq {}_{\ell}Z_n$ to be the image of $\Phi_n:  {}_{\ell}\Zt \to {}_{\ell}Z_n$.  That is, 
\[ {}_{\ell}Z_n' = \bigcap_{i \geq 0} \im (\phi^i:  {}_{\ell}Z_{n+i} \to {}_{\ell}Z_n).\]
Since ${}_{\ell}Z_n$ is noetherian and the $\phi_n$ are closed, this is a closed subscheme of ${}_{\ell}Z_n$, equal to $\im (\phi^{k}:  {}_{\ell}Z_{n+k} \to {}_{\ell}Z_n)$ for some $k$.  Similarly, let ${}_{\ell}Y_n' = \im (\Upsilon_n:  {}_{\ell}\Yt \to {}_{\ell}Y_n)$.  It is clear that ${}_{\ell}\Zt = \varprojlim {}_{\ell}Z_n'$ and ${}_{\ell}\Yt = \varprojlim {}_{\ell}Y'_n$.  We refer to ${}_{\ell}Z_n'$ and ${}_{\ell}Y_n'$ as  {\em essential}  point schemes, since modules in ${}_{\ell}Z_n'$ and ${}_{\ell}Y_n'$ are truncations of honest (shifted) point modules.

\begin{theorem}\label{thm:ptschemes}
Let $m$ be a \pp, and let $n \geq \ell \geq m$.   
\begin{enumerate}
 \item 
The morphism $s_n:  {}_{\ell}Z_n \to {}_{\ell} Y_n$ defined in Proposition~\ref{prop:immersion of point schemes} induces an isomorphism of essential point schemes
\[ s_n':  \xymatrix{{}_{\ell}Z'_n \ar[r]^{\cong} & {}_{\ell}Y'_n.}\]

\item The  limit
$s:  {}_{\ell}\Zt \to {}_{\ell}\Yt$
is an isomorphism of stacks.
\end{enumerate}
\end{theorem}
\begin{proof}
Since the subscript $\ell$ will remain fixed, we suppress it in the notation.     
Let $s_n'  := s_n |_{Z_n'}$.  It follows from commutativity of \eqref{snow} that $s_n'(Z_n')\subseteq Y'_n$.

We next prove (2).  
The limit $s = \varprojlim s_n$ is clearly a morphism of stacks --- that is, a natural transformation of functors.  Let $C$ be a commutative finite-dimensional local $\kk$-algebra.  We will show that $s$ is bijective on $C$-points.  

Let $y \in \Yt(C)$ be a $C$-point of $\Yt$, which by Theorem~\ref{fine moduli} corresponds to an exact sequence
\[ 0 \to J \to (S_C)_{\geq \ell} \to M \to 0,\]
where $M$ is an $\ell$-shifted $S_C$-point module.

For $i \geq \ell$, let $\sJ_i \subseteq p^* \sS_i$ be the subsheaf generated by $J_i \subseteq q_* p^* \sS_i$.
Let $\sJ := \bigoplus_{i \geq \ell} \sJ_i$.  
We will show that $\sM:= p^* \sS_{\geq \ell}/\sJ$ is an $\ell$-shifted $p^* \sS$-point module:  that is, that there is a $C$-point $y$ of $X$ so that $\sM_n \cong \sO_y \otimes p^* \sL_n$ for all $n \geq \ell$.

Since $p^* \sS_j$ is globally generated for $j \geq m$,   we have
$\sJ_i p^* \sS_j^{\sigma^i} \subseteq \sJ_{i+j}$ for $i \geq \ell$, $j \geq m$.  
Thus $ \sM$ is a coherent right module over the bimodule algebra $\sS' := \sO_{X_C} \oplus \bigoplus_{j \geq m } p^* \sS_j $.  
Further, each $\sM_j$ is clearly torsion over $X$, as $\Spec C$ is $0$-dimensional.    
As in \cite[Lemma~4.1(1)]{RS-0}, it follows from critical density of the orbits of the points in $P$ that there are a coherent $X$-torsion sheaf $\sF$ on $X_C$ and $n_0 \geq \ell$ so that $\sM_j \cong \sF \otimes_{X_C} p^* \sL_j$ for $j \geq n_0$.  
By critical density again, there is $n_1 \geq n_0$ so that
\[ \Supp(\sF) \cap p^{-1} \bigl( \sigma^{-j}(P) \bigr) = \emptyset\]
for $j \geq n_1$.  
This implies that for $j \geq n_1$ and $k \geq 1$, we have 
\[ \sM_j \cdot p^* \sS_k^{\sigma^j} =  \sM_j \otimes_{X_C} p^* \sS_k^{\sigma^j} = \sM_j \otimes _{X_C} p^* \sL_k^{\sigma^j}  \cong \sM_{j+k}\]
and $\sJ_j(p^* \sS_k^{\sigma^j}) = \sJ_{j+k}$.
  (In particular, $\sJ_{\geq n_1}$ and $\sM_{\geq n_1}$ are right $p^* \sS$-modules.)

By Lemma~\ref{lem-ample}, $\{p^* \sS_k^{\sigma^j}\}_{k \geq 0}$ is right ample on $X_C$.
As in \cite[Lemma~9.3]{RS}, for any $i \geq \ell$ and $k \gg 0$ we have 
$  q_*( \sJ_i p^* \sS_k^{\sigma^i})  =  J_i (S_C)_k \subseteq J_{i+k}$.  
For $i \geq n_1$ and for $k \geq 1$, we have $q_*(\sJ_i p^* \sS_k^{\sigma^i})= q_* \sJ_{i+k}\supseteq J_{i+k}$.  
There is thus $n_2 \geq n_1$ so that $J_j = q_*\sJ_j$ for $j \geq n_2$.

It follows from  Lemma~\ref{lem-ample} that there is $n_3 \geq n_2$ so that the top row of
\[
 \xymatrix@R=3pt{
0 \ar[r] & q_* \sJ_j\ar@{=}[d] \ar[r] & (S_C)_j \ar[r] & q_* \sM_j \ar@{=}[d] \ar[r] & 0 \\
& J_j && q_*( \sF \otimes_{X_C} p^* \sL_j)
}
\]
is exact for $j \geq n_3$.  
Thus $q_*( \sF \otimes_{X_C} p^* \sL_j) \cong (S_C)_j/J_j \cong C$ for $j \geq n_3$.  
Since  $\{p^* \sL_j\}$ is right ample, this implies that $\sF \cong \sO_y$ for some $C$-point $y$ of $X$.

For $\ell \leq i \leq n_3$, let $\sJ'_i := \sJ_i + \sI_y p^* \sS_i$.  
By choice of $\ell$, $\sJ'_i$ is 0-regular.  
Thus the rows of the commutative diagram
\beq\label{diagfoo}
 \xymatrix{
0 \ar[r] & J_i \ar[d]_{\subseteq} \ar[r] & (S_C)_i \ar@{=}[d] \ar[r] & C \ar[d]^{\alpha} \ar[r] & 0 \\
0 \ar[r] & q_* \sJ'_i \ar[r] & q_*p^* \sS_i \ar[r] &  q_* (p^* \sS_i/\sJ'_i)\ar[r] & 0
}
\eeq
are exact, and $\alpha$ is therefore surjective.  
This shows that, as a $(C \cong \sO_y)$-module, $p^* \sS_i/\sJ'_i$ is cyclic.

Let $N:= \Ann_C(p^* \sS_i/\sJ'_i)$.  
For $i \gg 0$, $p^* \sS_i/\sJ_i \cong \sO_y \otimes_{X_C} p^* \sL_i$ is killed by $\sI_y$.  
Thus for $i\gg 0$  we have $(\sJ'_i)_y = (\sJ_i)_y + (\sI_y p^* \sS_i)_y = (\sJ_i)_y$, and
\[ (\sJ_{i+j})_y \supseteq (\sJ_i)_y( p^* \sS_j^{\sigma^i})_y = (\sJ'_i )_y(p^* \sS_j^{\sigma^i})_y \supseteq 
N(p^* \sS_i)_y(p^* \sS_j^{\sigma^i})_y = 
 N( p^* \sS_{i+j})_y.\]
As $\Ann_C(p^* \sS_{i+j}/\sJ_{i+j}) = 0$ for $j \gg 0$, we must have $N = 0$.  

Thus in fact $p^* \sS_i /\sJ'_i \cong \sO_y \cong C$.  
Looking again at \eqref{diagfoo}, we see that $\alpha$ is an isomorphism, and so $J_i = q_*\sJ'_i$.  
As $\sJ'_i$ is 0-regular, it is globally generated:  in other words, $\sJ_i = \sJ'_i$.

We still need to show that $\sJ$ is a $p^* \sS$-module.  
Let $i \geq \ell$ and suppose that $\sJ_i p^* \sS_1^{\sigma^i} \not\subseteq \sJ_{i+1}$.   
Then $(\sJ_{i+1} + \sJ_i p^* \sS_1^{\sigma^i})/\sJ_{i+1}$ is a nonzero submodule of $p^* \sS_{i+1}/\sJ_{i+1} \cong \sO_y$.
Since $(\sJ_{i+1}+ \sJ_i p^* \sS_1^{\sigma^i}) \sS_j^{\sigma^{i+1}} \subseteq \sJ_{i+j+1}$ for all $j \gg 0$, a similar argument to the last paragraph but one produces a contradiction.  

Thus $p^* \sS_{\geq \ell}/\sJ$ is an $\ell$-shifted $C$-point module for $\sS$, and $q_*\sJ = J$.
This shows that $s:  \Zt \to \Yt$ induces a  surjection on $C$-points.   It follows from Corollary~\ref{cor-ptideal}(3) that $s$ is injective on $C$-points.

Consider the commutative diagram
\beq\label{foo} \xymatrix{
 \Zt \ar[r]^{s} \ar[d]_{\Phi_n} & \Yt \ar[d]^{\Upsilon_n} \\
Z'_n \ar[r]_{s'_n}  & Y'_n.}
\eeq
Let $x \in X$ be a $\kk$-point.  
There is some $N$ so that for $n \geq N$, $\Phi_n$ is a local isomorphism at all points in the preimage of $x$.  
We claim that for $n \geq N$, in fact all maps in \eqref{foo} are local isomorphisms at all points in the preimage of $x$.
In particular, $s$ is an isomorphism in the preimage of  $x$; since $x$ was arbitrary, $s$  is therefore an isomorphism of stacks.

So it suffices to prove the claim. 
We may work locally.  
If $z \in \Zt$ is a point lying over $x$, let $z_n $, $y$, $y_n$ be the images of $z$ in $Z'_n$, $\Yt$, and $Y'_n$, respectively.
Let  $B := \sO_{\Zt,z} \cong \sO_{Z'_n, z_n}$.
Let $A := \sO_{\Yt, y}$ and let $A' := \sO_{Y'_n, y_n}$.
We have  a commutative diagram of local homomorphisms
\[ \xymatrix{
 B \ar@{=}[d] & A \ar[l]_{s^\#} \\
B & A'. \ar[l]^{(s_n')^\#} \ar[u]_{\Upsilon_n^\#} }
\]

As $\Upsilon_n$ is scheme-theoretically surjective, $\Upsilon_n^{\#}$ is injective.  
It follows from Proposition~\ref{prop-likeRS} that $A$ is isomorphic to a local ring of some $Y'_m$ and in particular is a noetherian $\kk$-algebra.
By  Lemma~\ref{lem:s is scheme-theoretically surjective}, $s^{\#}$ is injective.  
Thus $(s_n')^{\#}$ is  injective.  
As $s_n$ is a closed immersion,  $(s_n')^{\#}$ is also surjective and thus an isomorphism; thus all maps in \eqref{foo} are local isomorphisms above $x$.
This proves the claim, as required.

(1).  Consider the diagram \eqref{foo}.   By Proposition~\ref{prop:immersion of point schemes}, $s_n:  Z_n \to Y_n$ is a closed immersion.  Thus the restriction $s'_n: Z'_n \to Y'_n$ is also  a closed immersion.

On the other hand, $Y'_n$ is the scheme-theoretic image of $\Upsilon_n$ and $s$ is an isomorphism.
Thus the composition $\Upsilon_n s = s_n' \Phi_n$ is scheme-theoretically surjective,  so $s_n'$ is scheme-theoretically surjective.
But a scheme-theoretically surjective closed immersion is an isomorphism.
 \end{proof}

Of course, the defining ideal of a 1-shifted  point module also defines a 0-shifted  point module, so if the \pp ~$m=1$, the conclusions of Theorem~\ref{thm:ptschemes} in fact hold for $m=0$.  In this situation, we will refer to $m=0$ as a \pp\!, by slight abuse of notation, since we need only the isomorphism ${}_{\ell} \Yt \cong {}_{\ell} \Zt$ for $\ell \geq m$ in the sequel. 

\begin{corollary}\label{cor:foobar}
 Let $\ell\in \NN$ and let ${}_{\ell} \Yt$ be the moduli stack of embedded $\ell$-shifted $S$-point modules, as above.  Then 
${}_{\ell}\Yt$ is a noetherian fpqc-algebraic stack.  
\end{corollary}
\begin{proof}
We know that ${}_{\ell}\Yt$ is a sheaf in the fpqc topology by Theorem~\ref{fine moduli}.   
Let $m$ be a \pp.  If $\ell \geq m$, then ${}_{\ell}\Yt\cong {}_{\ell} \Zt$ is noetherian fpqc-algebraic by Theorem~\ref{fine moduli}.

Suppose then that $\ell < m$.  
Let $T:  {}_{\ell} \Yt \to {}_m \Yt$ be the morphism defined by $T(M) := M_{\geq m}$.  
It is straightforward that $T$ is a morphism of functors, and that the product
\[ \Phi_m \times T : {}_{\ell} \Yt \to {}_{\ell} Y_m \times_X ({}_m \Yt)\]
is a closed immersion.
Since ${}_m \Yt$ is noetherian by the first paragraph,   it has an fpqc cover $U \to {}_m \Yt$ by a noetherian affine scheme $U$.  This cover can clearly be lifted and refined to induce an fpqc cover $V \to {}_\ell Y_m \times_X ({}_m \Yt)$ where $V$ is a noetherian affine scheme.   But then the Cartesian product
\[ 
\xymatrix{
V' \ar[r] \ar[d] \ar @{} [dr] |{\square} & {}_{\ell}\Yt \ar[d]^{\Phi_m \times T} \\
V \ar[r] & {}_{\ell}Y_m \times_X ({}_m \Zt)}
\]
gives an fpqc cover $V' \to \Yt$.  
Since $\Phi_m \times T$ is a closed immersion, so is $V' \to V$.  Thus $V'$ is isomorphic to 
 a closed subscheme of $V$ and is noetherian and affine. 
\end{proof}

Let $m$ be  a \pp\ and let $\ell \geq m$.  
We note that the relevant component of ${}_{\ell} \Yt \cong {}_{\ell} \Zt$ is the component containing the $\kk(X)$-point corresponding to the {\em generic point module} 
$\kk(X)z^{\ell} \oplus \kk(X)z^{\ell+1} \oplus \cdots$, 
which is isomorphic to 
$(Q_{\gr}(S))_{\geq \ell}.$

\section{A coarse moduli space for point modules}\label{COARSE MODULI}
In this section, we consider point modules up to module isomorphism in $\rqgr S$, and show that the scheme $X$ is a coarse moduli scheme for this functor.

We define the following maps. For any $\ell$, let $\Phi:  {}_{\ell}\Zt \to X$ be the map induced from the $f_n$.  Let $\Psi:  {}_0 \Zt \to {}_0 \Zt$ be the map induced from $\psi_n:  {}_0 Z_n \to {}_0 Z_{n-1}$.  Taking the limit of \eqref{d2}, we obtain that
\[ \Phi \Psi = \sigma \Phi:  {}_0 \Zt \to X.\]

For any noetherian commutative $\kk$-algebra $C$, there is  a graded $(\sO_{X_C}, \sigma \times 1)$-bimodule algebra $\sS_C$ given by pulling back $\sS$ along the projection map $X_C \to X$.  Taking global sections gives a functor
$\rGr \sS_C \to \rGr S_C.$
If $C=\kk$, this induces an equivalence
$\rqgr \sS_C \to \rqgr S_C$
by Theorem~\ref{thm-VdBSerre}.    In order to avoid the issues involved with extending this result to bimodule algebras over arbitrary base schemes, we work instead with point modules in $\rQgr \sS_C$ and $\rQgr S_C$.

Let $\ell \geq m$ where $m$ is a \pp\ (Definition \ref{positivity parameter}) and let $F$ be the moduli functor of (embedded) $\ell$-shifted point modules over $S$, as in the previous section.   
  Define an equivalence relation $\sim$ on $F(C)$ by saying that $M \sim N$ if (their images) are isomorphic in $\rQgr S_C$.  Define a contravariant functor $G:  \text{Affine schemes} \longrightarrow \text{Sets}$
by sheafifying, in the fpqc topology, the presheaf $G^{\on{pre}}$ of sets $\Spec C \mapsto  F(C)/\sim$.
Let $\mu:  F \to G$ be the natural map.  Likewise, let $\sh{F} \cong {}_{\ell} \Zt$ be the moduli functor of $\ell$-shifted point modules over $\sh{S}$, and let $\sh{G} := \sh{F}/\sim$, as above.  Let $\mu:  \sh{F} \to \sh{G}$ be the natural map.

We recall that $a_n: X_n \to X$ is the blowup of $X$ at $\sI_n$, and that by Corollary~\ref{cor-X-maps} there are morphisms $\alpha:  X_{n+1}\to X_n$ that intertwine with  the maps $a_n$.

We briefly discuss point modules over local rings.  
We note that if $C$ is a local ring of a point of $\Zt^o \cong \Xt$ with maximal ideal $\mf{m}$, then the map 
$h_{\Zt}(C) \to \sh{F}(C)$
has a particularly simple form.  Let $\zeta:  \Spec C \to \Xt$ be the induced morphism.  By critical density, there is some $n \geq m$ so that $\zeta_n(\mf{m})$ is not a fundamental point of any of the maps $\alpha^i:  X_{i+n} \to X_n$, for any $i > 0$.  Let $x_n := \zeta_n(\mf{m})$.  Define
\[ a_n^{-1} (\sh{S}_j) := a_n^{-1}(\sh{I}_j) \otimes_{X_n} a_n^* \Lsh_j.\]
Then $a_n^{-1}(\sh{S}_j)$ is flat at $x_n$ for all $j$.  Let
\[ \sh{M} := \bigoplus_{j \geq 0} a_n^{-1}(\sh{S}_j)_{x_n}.\]
Then $\sh{M}$ is flat over $C$ and is the $C$-point module corresponding to $\zeta$.   The $C$-action on $\sh{M}$ is obvious; to define the $\sh{S}$-action on $\sh{M}$, let $x := a_n(x_n)$.  Then there are maps
\[ \xymatrix{
\sh{M}_j \otimes_{\kk} \sh{S}_i(X) \ar[r] &
\sh{M}_j \otimes_{\kk} \sh{S}_i^{\sigma^j}(X) \ar[r] &
\sh{M}_j \otimes_{\struct_{X,x}} (\sh{S}_i^{\sigma^j})_x \ar[r] &
a_n^{-1}(\sh{S}_{j+i})_{x_n}. } \]
This gives a right $\sh{S}$-action on $\sh{M}$; by letting $C$ act naturally on the left and identifying $C$ with $C^{\op}$ we obtain an action of $\sh{S}_C$ on $M$.  

If $C$ is a local ring, we do not know if $S_C$ is necessarily  noetherian.  However, $C$-point modules in $\rQgr S_C$ are well-behaved, as follows.

\begin{lemma}\label{lem-local-torsion-cat}
Let $C$ be a commutative noetherian local $\kk$-algebra.  Let $N$ and $M$ be  $\ell$-shifted $C$-point modules, with $ M \cong  N $ in $\rQgr S_C$.  Then for some $k$ we have $M_{\geq k} \cong N_{\geq k}$.
\end{lemma}
\begin{proof}
The torsion submodules of $M$ and $N$ are trivial.  Thus
\[ \Hom_{\rQgr S_C}(M, N) = \varinjlim \Hom_{\rGr S_C}(M', N),\]
where the limit is taken over all submodules $M' \subseteq M$ with $M/M'$ torsion.  If $ M \cong  N$ in $\rQgr S_C$, then there is some submodule $M' \subseteq M$ so that $M/M'$ is torsion, and so that there is a homomorphism $f:  M' \to N$ so that  $\ker f$ and $N/f(M')$ are torsion.  Since $M$ and $N$ are torsion-free, we must have $\ker f = 0$.  

Thus it suffices to show that if $N$ is an $\ell$-shifted $C$-point module and $M \subseteq N$ is a graded submodule with $T:= N/M$ torsion, then $T_n=0$ for $n\gg 0$.  

Let $L$ be the residue field of $C$.  
 By assumption, $N$ is $C$-flat.  Thus there is an exact sequence
\[ 0 \to \Tor_1^C(T, L) \to M_L \to N_L \to T_L \to 0.\]
By \cite[Theorem~5.1]{ASZ1999}, the algebra $R_L$ is noetherian.  Thus $N_L$ and $T_L$ are  also noetherian.  Since $T_L$ is torsion, it is finite-dimensional.  Thus for $n \gg 0$, we have $(T_L)_n = (T_n)\otimes_C L = 0$.  By Nakayama's Lemma, $T_n = 0$.
\end{proof}

\begin{lemma}\label{lem-local-Psi}
Let $C$ be a noetherian local ring, and let $\sh{M}$, $\sh{N}$ be $C$-point modules over $\sh{S}$, corresponding to morphisms 
$f_{\sh{M}},f_{\sh{N}}: \Spec C \to {}_0 \Zt$.  If $\sh{N} \sim \sh{M}$, then there is some $k$ so that $\Psi^k f_{\sh{M}} = \Psi^k f_{\sh{N}}$.
\end{lemma}
\begin{proof}
Let $m$ be a \pp~ and let $M: = s(\sh{M}_{\geq m})$ and $N:=  s(\sh{N}_{\geq m})$.  Then $M \sim N$; by Lemma~\ref{lem-local-torsion-cat} there is some $k$, which we may take to be at least $m$, so that $M_{\geq k} \cong N_{\geq k}$.  Since ${}_k \Zt \cong {}_k \Yt$, we have $\sM_{\geq k} \cong \sN_{\geq k}$.  Thus the modules $\sh{M}[k]_{\geq 0} \cong \sh{N}[k]_{\geq 0}$ correspond to the same point of ${}_0 \Zt$; that is,  $\Psi^k f_{\sh{M}} = \Psi^k f_{\sh{N}}$.
\end{proof}

We now show that $X$ is a coarse moduli space for ${}_0\Zt/\sim$.  In fact, we prove this result in greater generality, to be able to analyze the spaces ${}_{\ell} \Yt$.  

\begin{proposition}\label{prop-coarse}
Let $\Zt:= {}_0\Zt$.  
Let $\Vt$ be a closed algebraic substack of $\Zt$ so that 
$\Xt \subseteq \Vt \subseteq \Zt$, and assume that $\displaystyle \Vt = \underset{\longleftarrow}{\lim}\; V_n$ where $V_n \subset Z_n$
is a closed subscheme that maps into $V_{n-1}$ under $Z_n\rightarrow Z_{n-1}$ for all $n$.  
Then $X$ is a coarse moduli space for $\Vt/\sim$.  More precisely, let $\sh{H}$ be the  image of $\Vt$ under $\mu:  \Zt \to \sh{G}$.   Then:
\begin{enumerate}
\item The morphism $\Phi: \Vt\to X$ factors via $\Vt\xrightarrow{\mu} \sh{H} \xrightarrow{t} X$.
\item Every morphism $\sh{H} \to A$ where $A$ is a scheme (of finite type) factors uniquely through $\sh{H} \xrightarrow{t} X$.
\end{enumerate}
\end{proposition}
\begin{proof}
(1)
It suffices to prove that if $C$ is a commutative noetherian ring and $\sh{M} \sim \sh{N}$ are $C$-point modules over $\sh{S}$, corresponding to maps $f_{\sh{M}}, f_{\sh{N}} : \Spec C \to \Zt$,  then we have $\Phi f_{\sh{M}} = \Phi f_{\sh{N}}:  \Spec C \to X$.
To show this,  it suffices to consider the case that $C$ is local.
  By Lemma~\ref{lem-local-Psi},  $\Psi^k f_{\sh{M}} = \Psi^k f_{\sh{N}}$ for some $k$.  Thus
\[ \Phi f_\sM = \sigma^{-k} \Phi \Psi^k f_\sM = \sigma^{-k} \Phi \Psi^k f_\sN = \Phi f_\sN,\]
as required.

(2)
Let $\nu : \sh{H} \to h_A$ be a natural transformation for some scheme $A$.  For all $n \in \NN$, let $P_n$ be the subscheme of $X$ defined by $\sI_n$.  
Fix any closed point $x \in X \smallsetminus \bigcup P_n$; some such $x$ exists since $\kk$ is uncountable.  Let $C := \struct_{X,x}$.  The induced map
$\Spec C \to \Xt \to \Vt$
gives a $C$-point module $\sh{M}_x$ as described above; its $\sim$-equivalence class is an  element of $\sh{H}(C)$.  Applying $\nu$, we therefore have a morphism 
$\Spec C  \to A.$
This extends to a morphism 
$g_x: U_x \to A$
for some open subset $U_x$ of $X$.  It follows from critical density of the orbits of points in $P$ that $X \smallsetminus \bigcup P_n$ is quasi-compact.  Thus we may take finitely many $U_x$, say $U_1, \ldots, U_k$, that cover $X \smallsetminus \bigcup P_n$, with maps $g_i: U_i \to A$.   
These maps all agree on the generic point of $X$ and so agree on overlaps $U_i \cap U_j$.

 Let 
$\displaystyle U:= \bigcup_{i=1}^k U_i,$
and let $g: U \to A$ be the induced map.
Then $X \ssm U \subseteq \bigcup P_n$ is a closed subset of $X$, and so 
$X \ssm U = \{z_{1}, \ldots, z_{r}\}$
for some $z_1, \ldots, z_r \in \bigcup P_n$.  
Let $n$ be such that  for any $i > 0$, the map $\phi^i:  Z_{n+i} \to Z_n$ is a local isomorphism at all points in the preimage of $\{z_{1} ,\ldots, z_{r}\}.$   

Let $\Phi_n:  \Zt \to Z_n$ be the map induced from the $\phi_m$.     
There is an induced map
$f_n^{-1}(U) \cap V_n \to U \to A.$
Further, for every $y \in \Vt \smallsetminus \Phi^{-1}(U)$, 
the map
$\Spec \sO_{\Vt, y} \to \Vt \to \sH \to A$
factors through $\Vt \to V_n$, as $\Phi_n$ is a local isomorphism at $y$. We thus obtain morphisms
$\Spec \sO_{V_n,y} \to A$
for all $y \in V_n$.  Using these, we may extend $g$ to give a morphism
  $\theta:  V_n \to A$
so that the diagram
\[ \xymatrix{
\Vt \ar[r]^{\mu} \ar[d]_{\Phi_n}		& \sH \ar[d]^{\nu} \\
V_n \ar[r]_{\theta} \ar@{-->}[d]_{f_n} 	& A \\
U \ar[ru]_{g} & 
}\]
commutes. 

We claim that $\theta$ contracts each of the loci $f_n^{-1}(z_{j})\cap V_n$ to a point.  To see this, let
$x, y \in \Phi^{-1}(z_{j}) \cap \Vt$, corresponding to maps
$f_x, f_y:  \Spec \kk \to \Vt$.  We must show that
$\theta \Phi_n f_x = \theta \Phi_n f_y$.

Since for $k \gg 0$, $\sigma^k(z_{j})$ is not in $\bigcup P_n$, $\Phi$ is a local isomorphism at $\Psi^k(\Phi^{-1}(z_{j}))$.  
We  have
\[ \Phi \Psi^k f_x =  \sigma^k \Phi f_x =  \sigma^k \Phi f_y = \Phi \Psi^k f_y\]
and so $\Psi^k f_x = \Psi^k f_y$.  
Therefore, $\mu f_x = \mu f_y$, and so
\[ \theta \Phi_n f_x = \nu \mu f_x = \nu \mu f_y = \theta \Phi_n f_y,\]
as we wanted.

The morphism $\theta:  V_n \to A$ thus factors set-theoretically to give a map from $X$ to $A$.  Since $X$ is smooth at all $z_i$ by critical density of the orbits of the $z_i$, it is well known that $\theta$ also factors scheme-theoretically.
\noindent
Consequently, we have the morphism $X \to A$ that we sought.  This proves Proposition \ref{prop-coarse}.
\end{proof}

\begin{theorem}\label{thm-coarse}
Fix a \pp~ $m$ (Definition \ref{positivity parameter}) and let $\ell\geq m$.  Then $X$ is a coarse moduli space for $G = F/\sim$. 
\end{theorem}

\begin{proof}
As above, we denote  the functor of $\ell$-shifted point modules over $\sh{S}$ modulo $\sim$ by $\sh{G}$.
By Theorem~\ref{thm:ptschemes}(2), it is  enough to show that $X$ is a coarse moduli space for $\sh{G}$.  Let $L:  {}_{\ell} \Zt \to {}_0 \Zt$ be the map that sends
$\sh{M} \mapsto \sh{M}[\ell]$.    Notice that if $\sh{M}$, $\sh{N}$ are $\ell$-shifted point modules, then  $\sh{M} \sim \sh{N}$ if and only if $\sh{M}[\ell] \sim \sh{N}[\ell]$.  That is, if we let  $\sh{G}'$ be the functor of (unshifted) point modules over $\sh{S}$ modulo $\sim$, then $L$ induces an inclusion $  \sh{G} \to \sh{G}'$ so that the diagram
\[ \xymatrix{
{}_{\ell}\Zt \ar[r]^L \ar[d]_{\mu} & {}_0 \Zt \ar[d]^{\mu} \\
\sh{G} \ar[r]_L & \sh{G}'
}
\]
commutes.  Let $V_n := \operatorname{Im}({}_{\ell}Z_{\ell+n}\rightarrow {}_0Z_n)$ and $\displaystyle \Vt := \underset{\longleftarrow}{\lim}\, V_n$, so $\Vt =  L({}_{\ell}\Zt)$.

Note that $L$ is injective on $\Xt \subseteq {}_{\ell}\Yt$.   Thus $\Vt$ satisfies the hypotheses of Proposition~\ref{prop-coarse}, and so $X$ is a coarse moduli scheme for $\sh{H} = \mu(\Vt) \cong \mu({}_{\ell} \Zt) \cong \sG$.   
\end{proof}

\bibliographystyle{amsalpha}

\providecommand{\bysame}{\leavevmode\hbox to3em{\hrulefill}\thinspace}
\providecommand{\MR}{\relax\ifhmode\unskip\space\fi MR }
\providecommand{\MRhref}[2]{%
  \href{http://www.ams.org/mathscinet-getitem?mr=#1}{#2}
}
\providecommand{\href}[2]{#2}


\end{document}